\documentclass[11pt,reqno]{amsart}
\usepackage{amsmath,amsfonts,amsthm,color,amssymb, mathrsfs}
\usepackage{xcolor}
\usepackage[T1]{fontenc}
\usepackage{wasysym}
\usepackage[normalem]{ulem}
\usepackage{stmaryrd}
\usepackage[left=1in, right=1in, top=1.1in,bottom=1.1in]{geometry}
\usepackage{enumitem}
\usepackage{hyperref}
\usepackage{cancel}
\usepackage{comment}
\usepackage{csquotes}
\usepackage{graphicx}
\usepackage{pstricks}
\usepackage{lmodern}

\hypersetup{
	colorlinks   = true,
	citecolor    = blue,
	linkcolor=blue
}
\allowdisplaybreaks
\newtheorem{theorem}{Theorem}[section]
\newtheorem{lemma}{Lemma}

\newtheorem{definition}{Definition}

\newtheorem{remark}{Remark}

\def\la{\left\langle}
\def\ra{\right\rangle}

\def\R{\mathbb R}
\def\E{\mathbb E}

\newcounter{bean}
\newcommand{\benuma}{\setlength{\labelwidth}{.25in}
	
	\begin{list}
		{(\alph{bean})}{\usecounter{bean}}}
\newcommand{\eenuma}{\end{list}}

\begin{document}
\title[ ]{Maximum principle for recursive optimal control problem of stochastic delay evolution equations}
\author[G. Liu]{ Guomin Liu}
\address{School of Mathematical Sciences, Nankai University, Tianjin 300071,
China}
\email{gmliu@nankai.edu.cn}
\author[J. Song]{Jian Song}
\address{Research Center for Mathematics and Interdisciplinary Sciences,
Shandong University, Qingdao 266237, China}
\email{txjsong@sdu.edu.cn}

\author[M. Wang]{Meng Wang}
\address{School of Mathematics, Shandong University, Jinan 250100, Shandong,
China}
\email{wangmeng22@mail.sdu.edu.cn}
\date{}
\maketitle

\begin{abstract}
For a class of stochastic delay evolution equations driven by cylindrical $Q$-Wiener process, we study the Pontryagin's maximum principle for the stochastic recursive
optimal control problem. The delays 
are given as moving averages with
respect to  general finite measures and  appear  in all the coefficients of the control system. In particular, the final cost can contain the state delay. To derive the main result, we introduce
a new form of anticipated backward stochastic evolution equations with terminals acting on an interval as the adjoint equations of the delayed state equations  and
deploy a proper dual analysis between  them. Under certain convex assumption on the coefficient function and the Hamiltonian, we also
show sufficiency of the maximum principle.

 \smallskip

\end{abstract}

\tableofcontents

\vspace{-10pt}

\section{Introduction}

In this paper, we investigate a stochastic recursive optimal control problem
of stochastic delayed evolution equation (SDEE for short)  evolving in a Hilbert space $H$:
\begin{equation}
\left\{ \begin{aligned} dx(t)= &\,\, \left[A(t)x(t)+b
\left(t,x(t),\int_{-\delta
}^{0}x(t+s)m(ds),u(t),\int_{-\delta }^{0}u(t+s)m(ds)\right)\right]dt \\ &
+\left[B(t)x(t)+\sigma \left(t,x(t),\int_{-\delta }^{0}x(t+s)m(ds),u(t),\int_{-\delta
}^{0}u(t+s)m(ds)\right)\right]dw(t),\, t\in(0,T], \\ x(t)= &\,\, x_{0}(t),\,\,u(t)=u_0(t),\quad t\in
\lbrack -\delta ,0], \end{aligned}\right.  \label{EQ-1}
\end{equation}%
with the cost functional $J(u(\cdot )):=y(0)$, where $(y,z)$ solves the following anticipated backward
stochastic differential equation (ABSDE for short) 
\begin{equation}
\left\{ \begin{aligned} -dy(t)= &\,\, \mathbb{E}^{\mathcal{F}%
		_{t}}\Bigg[f\Bigg(t,x(t),\int_{-\delta
}^{0}x(t+s)m(ds),y(t),\int_{-\delta }^{0}y(t-s)m(ds),z(t),\int_{-\delta
}^{0}z(t-s)m(ds),\\ &\ \ \ \qquad \qquad  u(t),\int_{-\delta }^{0}u(t+s)m(ds)\Bigg)\Bigg]dt
-z(t)dw(t), \, t\in[0,T), \\ y(T)= & \,\,h\left(x(T),\int_{-\delta}^0x(T+s)\nu(ds)\right), \ y(t)=0,\ z(t)=0,\ t\in (T,T+\delta ]. \end{aligned}\right.  \label{Myeq2-17}
\end{equation}%
Here, $w(\cdot )$ is a  \emph{cylindrical $Q$-Wiener process} on  some
Hilbert space $K$, $A(t)$ and $B(t)$ are random (unbounded) linear operators, the mappings $b,\sigma ,h,$ and $f$ are coefficient functions taking values in $H$ or $\mathcal L(K;H)$ depending on the context, $u$ is a control process with values in $U$ which is  a convex subset of a Hilbert space $H_1$, and $m,\nu$ are
finite measures on $[-\delta ,0].$  Note that the cylindrical $Q$-Wiener process $w$ can be a classical finite dimensional Wiener process, a standard (i.e.,  trace-class) Hilbert space valued  Wiener process, or a generalized Hilbert space valued  Wiener process (e.g. a cylindrical Wiener process), by choosing the Hilbert space $K$ and the covariance operator $Q$ properly. The significance of ABSDEs allows the models to not only reflect the current available information but also the forecasts of future  changes and trends, offering more accurate and dynamic  optimal decision-making.

In classical optimal control theory, the performance of the
control strategy usually is  characterized by a cost functional (utility function) which
consists of a terminal cost and a running cost. Duffie and Epstein \cite{duffie1992stochastic}
introduced the notion of stochastic differential recursive utility, which was  later extended 
to the form of backward stochastic differential equation (BSDE for short) by  Peng \cite{peng1993backward}, El
Karoui, Peng and Quenez \cite{el1997backward}. An 
optimal control problem with cost functional  described by a BSDE is then called a stochastic recursive optimal control problem.

The Pontryagin's maximum principle is widely recognized as an effective approach in solving optimal control problems (see \cite%
{peng1990general,doi:10.1137/S0363012996313100,WU20131473,Yong2010optimality}
and the references therein). In particular, for the
stochastic recursive optimal control problem concerning finite dimensional systems,  Peng \cite{peng1993backward} 
established a local form of the stochastic maximum principle.
As for infinite dimensional control systems, Benssousssan \cite
{bensoussan2006lectures}, Hu and Peng \cite{hu1990maximum} studied the conventional (non-recursive) utility case; for other developments
on this direction, we refer to \cite%
{du2013maximum,lu2014general,fuhrman2013stochastic,SW21,liu2021maximum, SongWang21}.

 In order to describe past-dependent phenomena which exist widely in reality, delay terms were added  into  control systems. The adjoint equations of delay systems are now known as  anticipated BSDEs (ABSDEs for short), the theory of  which was established by Peng and Yang  \cite{10py}. Since then, a fruitful
literature has been developed on the study of maximum principle for
finite dimensional delay systems. In particular, {\O }ksendal, Sulem and Zhang \cite%
{11osz} studied the system with pointwise delay and moving average delay with
respect to Lebesgue measure in the  state.  Chen and Wu \cite{10cw}
considered the system where the state equation has pointwise delay in both state and
 control. Huang and Shi \cite{HS12} investigated the 
system of a fully coupled forward-backward stochastic differential with
pointwise delay in both state and control. Guatteri and Masiero~\cite{guatteri2021stochastic} studied the case in which  the delays in both state
and control are in the form of a integral with respect to a general finite regular measure. More related literature can be found in \cite%
{FMT10,GM23,20lww,MSWZ23,YU20122420,ZX17}.

In the context of infinite dimensional delay systems with conventional
utilities, {\O }ksendal, Sulem and Zhang \cite{oksendal2012optimal} obtained
the maximum principle for stochastic partial differential equations (SPDEs) with jumps, both  pointwise delay
and moving average delay  in the state being given by an integral with respect to the Lebesgue measure. Meng and Shen~\cite{16ms} established the maximum principle for stochastic
evolution equations (SEEs) with pointwise delay in both  state and
control; Dai, Zhou and Li \cite{DZL21} studied the infinite
horizon case under this setting. Guatteri, Masiero and Orrieri \cite{guatteri2017stochastic}
obtained the maximum principle for a class of SPDEs where the delay is an
integral with respect to a finite regular measure and the final cost also contains a state delay term. To the best of our knowledge, this is the first and only work involving the dependence on the past trajectory in the final cost in the infinite dimensional case. To handle  this new delay term, they introduced a type of  anticipated backward
stochastic evolution equations (ABSEEs) with a datum acting on an interval.   We remark  that in the state equation of \cite{guatteri2017stochastic}, the drift term does not involve the control delay and the diffusion  is independent of 
state and control (consequently, the  ABSEE is linear and does not contain $q$ in the generator). The ABSEE was studied in \cite{guatteri2017stochastic} by the duality method and 
weak approximation of measures, assuming that the measures  are regular and the datum is continuous.  Note also that  all the above-mentioned  control systems  do not contain unbounded operators in 
diffusion terms.   We also refer to \cite{hu1996maximum,guatteri2021stochastic,MSWZ23}
for related results on finite dimensional delay systems. 
 

The objective of this paper is to study the stochastic maximum principle for
the recursive optimal control problem of the infinite dimensional stochastic system  with delays (see \eqref{EQ-1} and 
\eqref{Myeq2-17}). In our controlled system, the delays (of the
control and the state) and the unbounded operators appear in  both  drift
and diffusion terms, and the final cost term can also depend on the past history of the state. The delays take the form of an integral  with respect to a general
finite measure. Note that we do not require the measures to be regular or the datum to be continuous, both of  which was assumed in \cite{guatteri2017stochastic}.  It is also worth noting that the noise in our setting is quite general, covering both finite dimensional Wiener processes as well as infinite dimensional Wiener processes such as trace-class and cylindrical Wiener processes. 

We conclude our introduction by making some remarks on our result and strategy which is different from the known literature.  In the derivation of the maximum principle, the adjoint equation turns out to be an ABSEE with a terminal datum acting on an interval as shown in~\cite{guatteri2017stochastic}. We propose a new  approach to study the wellposedness of ABSEEs  of a general form (see~\eqref{absee-0}), where  $m$ and $dF$ are finite  (signed) measures without assuming the regularity, the datum $\zeta$  is  measurable, and the generator contains unbounded operators and is nonlinear  in both  $p$ and $q$. Our strategy is as follows. First note that the solution $p$ of \eqref{absee-0} may be discontinuous, since it is possible that $dF$  is a singular  measure. Thus,  we shall employ the version of It\^{o}'s formula for general (discontinuous) semimartingales  given  by \cite{82gk}, in order to carry out calculations when deriving a priori estimates for the difference of the solutions of ABSEEs. Then we apply the  continuation method to obtain the wellposedness result for the case when the datum $\zeta$ is identically zero, and the case when $\zeta$ takes values in $V$ can be solved by the method of solution translation (note that we are not allowed to apply this method to ABSEE \eqref{absee-0} with $H$-valued $\zeta$ due to   the appearance of unbounded operators). Finally, the general wellposedness result for ABSEE with $H$-valued $\zeta$ is obtained by an approximation argument based on the a priori estimates. In the derivation of the stochastic maximum principle, the delay terms make the the dual analysis highly nontrivial. Some estimations and a duality equality regarding the delay
measures play  an important role in the analysis (see Lemmas \ref{Le3-2}, \ref{Le3-3} and \ref{Le3-1}). 

The rest of this paper is organized as follows. We recall some preliminary
results on infinite dimensional stochastic analysis in Section \ref{sec:pre}. In Section
\ref{sec:DSEE}, we present the well-posedness results of delayed SEEs and anticipated
BSEEs. We formulate our recursive optimal control problem and derive the
maximum principle in Section \ref{sec:SMP}. Finally, Section \ref{sec:APP} provides some
applications on linear quadratic (LQ) problem and controlled SPDEs and the Appendix includes a complimentary proof.

\section{Preliminaries}\label{sec:pre}

In this section, we provide some preliminaries on stochastic calculus in  infinite dimensional spaces. We refer to \cite{DZ92,PR07} for more details.

Let $X$ and $Y$ be generic Banach spaces.  We denote by $\mathcal{L}(X,Y)$ the space
of bounded linear operators mapping from $X$ to $Y$, and we write $\mathcal{L}(X)$ for $\mathcal{L}(X,X)$  and denote by $I_{X}$ the
identity operator. If $X$ is a separable Hilbert space with orthonormal basis $\{e_j\}_{j=1}^N$, where   $N\in \mathbb N\cup\{\infty\}$  is a  finite number or infinity depending on the dimension of  $X$. In the rest of this paper, we consider the case $N=\infty$ and the result also holds for the  finite dimensional case. 
We denote by $\mathcal{L}_{2}(X,Y)$ the space of \emph{Hilbert-Schmidt operators} mapping from $X$ to $Y,$ i.e., $\mathcal{L}_{2}(X,Y)$ consists of  $T\in \mathcal{L}(X,Y)$ satisfying 
\begin{equation*}\label{e:norm-L2}
    \Vert T\Vert _{\mathcal{L}_{2}(X,Y)}^{2}:=\sum_{j=1}^{\infty}\Vert Te_{j}\Vert _{Y}^{2}<\infty.
\end{equation*}
If we assume further that $Y$ is a separable Hilbert space, the space $\mathcal{L}_{2}(X,Y)$ of Hilbert-Schmidt operators becomes a separable Hilbert space with the inner product
\begin{equation*}
\left\langle T,G\right\rangle _{\mathcal{L}_{2}(X,Y)}:=\sum_{j=1}^{\infty}\left
\langle Te_{j},Ge_{j}\right\rangle _{Y}.
\end{equation*}

Now we recall some facts on  (generalized) Hilbert space valued Wiener processes in some complete probability space $(\Omega ,\mathcal{F},P)$. Let $K$ be a
separable Hilbert space with an orthonormal basis $\{e_{j}\}_{j=1}^{\infty}$ 
 and $Q\in\mathcal L(K)$ be a nonnegative definite and symmetric (i.e. self-adjoint) operator which will serve as the covariance operator of the Wiener process $w=\{w(t)\}_{t\in[0,T]}$.  
 
 If $Q$  has finite trace, i.e., 
\[\text{tr}\,Q:= \sum_{j=1}^\infty \langle Qe_j, e_j\rangle_K=\sum_{j=1}^\infty \langle Q^{\frac12}e_j, Q^{\frac12}e_j\rangle_K<\infty,\]
where $Q^{\frac12}$ denotes the nonnegative square root of $Q$,  there exists an orthonormal basis  still denoted by $\{e_j\}_{j=1}^\infty$ which diagonalizes $Q$, i.e., \[Qe_j=\lambda_j e_j, ~j\in\mathbb N,\] for $\lambda_j\ge0$ being the eigenvalues.  In this case, the \emph{$Q$-Wiener process} $w$ is a \emph{standard} $K$-valued Wiener process and can be expressed by  
\begin{equation*}
w(t)=\sum_{j=1}^{\infty}\beta ^{j}(t)Q^{\frac12}e_{j}=\sum_{j=1}^{\infty}\beta ^{j}(t)%
\lambda _{j}^{\frac12}e_{j},
\end{equation*}
where $\{\beta ^{j}(t)\}_{j\in\mathbb N}$ is a family of independent one-dimensional standard Brownian
motions.  

If $\text{tr}\,Q=\infty$, the Wiener process  \[ w(t) = \sum_{j=1}^\infty \beta^{j}(t)Q^{\frac12} e_j \]
 is a so-called \emph{cylindrical $Q$-Wiener process}. Note that the series does not converge in $K$ and hence $w$ is a \emph{generalized} $K$-valued process 
which
can be charactrerized by the centered Gaussian family 
\[\Big\{ w_t(\varphi):=\langle w(t), \varphi\rangle_K = \sum_{j=1}^\infty \langle Q^{\frac12} e_j, \varphi \rangle_K \beta^j(t), t\in[0,T], \varphi\in K\Big\} \]
with covariance function
\[\E[w_t(\varphi)w_s(\psi)]=(t\wedge s)\langle Q\varphi, \psi\rangle_K=(t\wedge s)\langle Q^{\frac12}\varphi, Q^{\frac12}\psi\rangle_K.\]
In particular, $w$ is called \emph{cylindrical Wiener process} when $Q=I_K$ is the identity operator.

 Let $\mathbb{F}=\{\mathcal{F}_{t}\}_{t\geq 0}$
be the filtration generated by the Wiener process $\{w(t)\}_{t\in \lbrack 0,T]}$ and augmented
by the class of all $P$-null sets of $\mathcal{F}$. For $t\in \lbrack
-\delta ,0)$, we define $\mathcal{F}_{t}:=\mathcal{F}_{0}$. Let $E$  denote a generic separable
Hilbert space with norm $\Vert \cdot \Vert _{E}$. We introduce the
following spaces that will be used in the paper.

\begin{itemize}
\item For any $\sigma $-algebra $\mathcal{G}$, $L^{2}(\mathcal{G};E)$ is the set of all $\mathcal{G}$-measurable random
variables $\xi $ taking values in $E$ such that 
\begin{equation*}
\mathbb{E}\left[ \Vert \xi \Vert _{E}^{2}\right] <\infty;
\end{equation*}
\item $L^{2}(0,T;E)$ denotes the set of all $E$-valued deterministic
processes $\phi =\{\phi (t),\ t\in \lbrack 0,T]\}$ such that 
\begin{equation*}
\int_{0}^{T}\Vert \phi (t)\Vert _{E}^{2}dt<\infty ;
\end{equation*}
\item $L_{\mathbb{F}}^{2}(0,T;E)$ denotes the set of all $E$-valued $%
\mathbb{F}$-adapted processes $\phi =\{\phi (t,\omega ),\ (t,\omega )\in
\lbrack 0,T]\times \Omega \}$ such that 
\begin{equation*}
\mathbb{E}\Big[ \int_{0}^{T}\Vert \phi (t)\Vert _{E}^{2}dt\Big] <\infty ;
\end{equation*}
\item $D_{\mathbb{F}}^{2}(0,T;E)$ is the set of all $E$-valued $%
\mathbb{F}$-adapted c\`{a}dl\`{a}g processes $\phi =\{\phi (t,\omega ),\ (t,\omega )\in
\lbrack 0,T]\times \Omega \}$ such that 
\begin{equation*}
\mathbb{E}\Big[ \sup_{0\leq t\leq T}\Vert \phi (t)\Vert _{E}^{2}\Big] <\infty.
\end{equation*}
\item Given a finite-variation process $F$ on $[0,T]$, $L_{\mathbb{F},F}^{2}(0,T;E)$ denotes the set of all $E$-valued progressively measurable processes $\phi$ satisfying
\begin{equation*}
\mathbb{E}\Big[ \int_{0}^{T}\Vert \phi (t)\Vert _{E}^{2}d|F|_v(t)\Big] <\infty,
\end{equation*}
 where $|F|_v$ is the total variation process of $F$. In particular, when $F(t)=t$, $L_{\mathbb{F},F}^{2}(0,T;E)$ coincides with $L_{\mathbb{F}}^{2}(0,T;E).$
\end{itemize}

Let $V$ and $H$ be two real separable Hilbert spaces such that $V$ is
densely embedded in $H$. Identify $H$ with its dual space $H^{\ast }$ and
denote by $V^{\ast }$ the dual space of $V$. Then we have $V\subset
H=H^{\ast }\subset V^{\ast }$. Denote by $\left\langle \cdot ,\cdot
\right\rangle _{H}$ (resp. $\left\langle \cdot ,\cdot \right\rangle _{\ast }$%
) the scalar product in $H$ ( the duality product between $V^{\ast }$ and $V$). We call $%
(V,H,V^{\ast })$ a \emph{Gelfand triple}.

Recall that $K$ is the Hilbert space where the Wiener process $w$ takes values. Then its subspace $K_{0}:=Q^{\frac12}(K)$ is a Hilbert space endowed with the inner product
\begin{equation*}
\left\langle u,v\right\rangle _{0}=\langle Q^{-\frac{1}{2}}u,Q^{-\frac{1}{2}%
}v\rangle _{K},\text{ }u,v\in K_{0}.
\end{equation*}
Denote by $\mathcal{L}_{2}^{0}(K,H):=\mathcal{L}_{2}(K_0,H)=\mathcal{L}_{2}(Q^{\frac{1}{2}
}(K),H)$ of which the norm is given by \[\Vert F\Vert _{\mathcal{L}
_{2}^{0}(K,H)}:=\Vert F\Vert _{\mathcal{L}_{2}(K_0,H)}=\Vert
FQ^{\frac{1}{2}}\Vert _{\mathcal{L}_{2}(K,H)}.\] 
We also write $
\mathcal{L}_{2}^{0}$ for $\mathcal L_2^0(K,H)$  for notation simplicity. For $f\in L_{\mathbb{F}}^{2}(0,T;\mathcal{L}%
_{2}^{0}),$ we define the stochastic integral with respect to $w$ as follows:%
\begin{equation*}
\int_{0}^{T}f(t)dw(t):=\sum_{k=1}^{\infty}
\int_{0}^{T}f(t) Q^{\frac12}e_{k}d\beta ^{k}(t),
\end{equation*}
where the right-hand side is understood as a limit in $L^{2}(\mathcal{F}_{T};H)$. The process $\int_{0}^{t}f(s)dw(s)$ is an $H$-valued continuous martingale with
\begin{equation*}
\mathbb{E}\left[\left\Vert \int_{0}^{t}f(s)dw(s)\right\Vert _{H}^{2}\right]=\mathbb{E}\Big[%
\int_{0}^{t}\Vert f(s)\Vert _{\mathcal{L}_{2}^{0}(K,H)}^{2}ds\Big]=\mathbb{E}%
\Big[\int_{0}^{t}\Vert f(s)Q^{\frac12}\Vert _{\mathcal{L}_{2}(K,H)}^{2}ds\Big].
\end{equation*}%
The following Burkholder-Davis-Gundy inequality holds:
\begin{equation*}
\mathbb{E}\left[\sup_{0\leq t\leq T}\left\Vert \int_{0}^{t}f(s)dw(s)\right\Vert _{H}^{2}
\right]\leq C\mathbb{E}\Big[\int_{0}^{T}\Vert f(t)\Vert _{\mathcal{L}%
_{2}^{0}(K,H)}^{2}dt\Big]=C\mathbb{E}\Big[\int_{0}^{T}\Vert f(t)Q^{\frac12}\Vert _{%
\mathcal{L}_{2}(K,H)}^{2}dt\Big].
\end{equation*}

 Moreover, given $f \in L^2_{\mathbb F}(0,T;\mathcal L_2^0)$ and $g \in L^2_{\mathbb F}(0,T;H)$, the real-valued martingale $\int_{0}^{t} \langle f(s) dw(s), g(s)\rangle_H$ is defined by 
\[ M_t: = \int_{0}^{t} \langle f(s) dw(s), g(s)\rangle_H=\sum_{k=1}^{\infty}
\int_{0}^{t}\langle f(s) Q^{\frac12}e_{k}, g(s)\rangle_H d\beta ^{k}(s),
\]
of which the quadratic variation is 
\begin{equation}\label{e:QV}
\begin{aligned}
 &       \langle M \rangle_t = \sum_{k=1}^\infty  \int_{0}^{t}\langle f(s) Q^{\frac12}e_{k}, g(s)\rangle_H^2 ds \\& \le \int_{0}^{t} \sum_{k=1}^\infty  
 \|f(s) Q^{\frac12}e_{k}\|_H^2\| g(s)\|_H^2 ds =\int_0^t \|f(s)\|^2_{\mathcal L_2^0} \|g(s)\|_H^2ds.
 \end{aligned}
\end{equation}

Consider three processes $\{v(t,\omega ),\ (t,\omega )\in \lbrack 0,T]\times
\Omega \}$, $\{M(t,\omega ),\ (t,\omega )\in \lbrack 0,T]\times \Omega \}$, and $%
\{v^{\ast }(t,\omega ),\ (t,\omega )\in \lbrack 0,T]\times \Omega \}$ with
values in $V$, $H$ and $V^{\ast }$, respectively. Let $v(t,\omega )$ be
measurable with respect to $(t,\omega )$ and be $\mathcal{F}_{t}$-measurable
with respect to $\omega $ for  $t\in \lbrack 0,T]$, the quantity $\left\langle v^{\ast }(t,\omega ),\eta \right\rangle _{\ast
} $ be measurable with respect to $(t,\omega )$ and $\mathcal{F}_{t}$%
-measurable with respect to $\omega $ for $t\in \lbrack 0,T]$ for any $\eta \in V$, $M(t,\omega )$ be  a continuous
local martingale. Let $\left\langle M\right\rangle $ denote
the increasing process part for $\Vert M\Vert _{H}^{2}$ in the Doob-Meyer
decomposition. Suppose $F(t)$ is a real-valued adapted  c\`{a}dl\`{a}g finite-variation process on $[0,T]$ and $\zeta(t)$ is an $H$-valued adapted process on $[0,T]$, and $\zeta\in L_{\mathbb{F},F}^{2}(0,T;H)$, $v^{\ast }\in L_{\mathbb{F}}^{2}(0,T;V^{\ast })$, $v\in L_{\mathbb{F}}^{2}(0,T;V)$.

The following It\^o's formula is borrowed from \cite{82gk}.

\begin{lemma}
\label{Ito-lemma} Suppose that for
each $\varphi \in V$, it holds  \begin{equation*}
\left\langle v(t),\varphi \right\rangle _{H}=\int_{0}^{t}\left\langle v^{\ast }(s),\varphi
\right\rangle _{\ast }ds+\int_{(0,t]}\left\langle \zeta(s),\varphi
\right\rangle dF(s)+\left\langle M(t),\varphi \right\rangle _{H},
\end{equation*}
for $dt\times dP$-almost all $(t,\omega )\in \lbrack 0,T]\times
\Omega $. Then there exists an adapted c\`{a}dl\`{a}g $H$-valued process $h(\cdot )$ such that

\begin{itemize}
\item[(i)] 
for $dt\times dP$-almost all $(t,\omega )\in \lbrack 0,T]\times \Omega $, $h(t, \omega)=v(t,\omega)$;

\item[(ii)] for $t\in \lbrack 0,T]$, it holds almost surely,
\begin{equation}\label{e:ito-lemma}
    \begin{aligned}
\Vert h(t)\Vert _{H}^{2}=&\Vert h(0)\Vert _{H}^{2}+2\int_{0}^{t}\left\langle
v^{\ast }(s),v(s)\right\rangle _{\ast }ds+2\int_{(0,t]}\left\langle
h(s),\zeta(s)\right\rangle dF(s)\\
&+2\int_{0}^{t}\left\langle
h(s),dM(s)\right\rangle _{H}+\left\langle M\right\rangle (t)-\int_{(0,t]}\Vert \zeta(s)\Vert _{H}^{2}\Delta F(s)dF(s),
\end{aligned}
\end{equation}
 where $\Delta F(s) = F(s) -F(s^-)$. 
\end{itemize}
\end{lemma}
\begin{proof}
When $v(\cdot)$ is a $V$-valued process such  that  for each $\varphi\in V$, it holds for  $dt\times dP$-almost all $(t,\omega)\in[0,T]\times \Omega$, 
\[\left\langle v(t),\varphi \right\rangle _{H}=\int_{0}^{t}\left\langle v^{\ast }(s),\varphi
\right\rangle _{\ast }ds+\left\langle N(t),\varphi \right\rangle _{H},
\]
where $v^*(\cdot)$ is a  $V^*$-valued process and $N(\cdot)$ is an $H$-valued locally square integrable c\`adl\`ag martingale,  the It\^o's formula  was proved in  \cite[Theorem 1]{82gk}.  The desired result can be obtained by same argument used in the proof of \cite[Theorem 1]{82gk}, with the $H$-valued  c\`adl\`ag martingale $N(t)$ being replaced by  the $H$-valued  c\`adl\`ag semi-martingale $\int_{(0,t]} \zeta(s) dF(s) + M(t)$. 
\end{proof}

\section{SDEEs and anticipated BSEEs}\label{sec:DSEE}

\subsection{Stochastic delay evolution equations}
Recall that $(V, H, V^*)$ is a Gelfand triple, and $K$ is a separable Hilbert space where the $Q$-Wiener process $w$ takes values. Also recall the notation $\mathcal L_2^0=\mathcal L_2(Q^{\frac12}(K), H)$.  Given the linear operators 
\begin{equation*}
A:[0,T]\times \Omega \rightarrow \mathcal{L}(V,V^{\ast
}),\quad B:[0,T]\times \Omega \rightarrow \mathcal{L}(V,\mathcal{L}%
_{2}^{0}),
\end{equation*}%
and the nonlinear functions 
\begin{equation*}
b:[0,T]\times \Omega \times H\times H\rightarrow H,\quad \sigma
:[0,T]\times \Omega \times H\times H\rightarrow \mathcal{L}_{2}^{0},
\end{equation*}%
we consider the following stochastic delay evolution equation (SDEE) in $%
(V,H,V^{\ast })$: 
\begin{equation}
\left\{
\begin{aligned}
dx(t)=& \left[A(t)x(t)+b(t,x(t),\int_{-\delta }^{0}x(t+s)m(ds))\right]dt \\
& +\left[B(t)x(t)+\sigma (t,x(t),\int_{-\delta }^{0}x(t+s)m(ds))\right]dw(t),\
\ t\in \lbrack 0,T], \\
x(t)=& x_{0}(t),\,\,t\in \lbrack -\delta ,0],
\end{aligned}%
\right.
\label{sdee-0}
\end{equation}%
where $\delta \in \lbrack 0,T]$ is a time delay parameter, $m(ds)$ is a
finite measure on $[-\delta ,0]$, $x_{0}:[-\delta ,0]\rightarrow
H $ is the initial path, and $w$ is a (cylindrical) $K$-valued  $Q$-Wiener process. For simplicity of presentation, for $t\in \lbrack
0,T] $ we set 
\begin{equation}
\eta _{\delta }(t):=\int_{-\delta }^{0}\eta (t+s)m(ds):=\int_{[-\delta
,0]}\eta (t+s)m(ds),  \label{Myeq3-3}
\end{equation}%
whenever the integration of the function $\eta (t+\cdot ):[-\delta
,0]\rightarrow E$ with respect to the measure $m$ exists (recall that $E$ represents a generic separable Hilbert space). Then 
\begin{equation*}
x_{\delta }(t)=\int_{-\delta }^{0}x(t+s)m(ds).
\end{equation*}%
and (\ref{sdee-0}) can be written as 
\begin{equation}\label{sdee}
\left\{
\begin{aligned}
dx(t)=& \big[A(t)x(t)+b(t,x(t),x_{\delta }(t))\big]dt \\
& +\big[B(t)x(t)+\sigma (t,x(t),x_{\delta }(t))\big]dw(t),\ \ t\in \lbrack
0,T], \\
x(t)=& x_{0}(t),\,\,t\in \lbrack -\delta ,0].
\end{aligned}%
\right.
\end{equation}

\begin{remark}
In general, since $m$ is a generic finite measure, for any two real numbers $a\leq
b,$ the integral with respect to $m$ may be different on intervals such as
$[a,b],(a,b],[a,b)$ and $(a,b)$. In this paper, for notational simplicity,
we will use $\int_{a}^{b}$ as above to denote the integrals over the
closed interval $[a,b]$.
\end{remark}

Throughout the rest of the paper, we denote by $C$ a generic positive constant
which may differ line by line.

\begin{definition}\label{def:SDEE}
A process $x(\cdot )\in L_{\mathbb{F}}^{2}(-\delta ,T;V)$ is called a
solution to \eqref{sdee}, if for $dt\times dP$-almost all $(t,\omega )\in \lbrack -\delta
,T]\times \Omega ,$ it holds in $V^{\ast }$ that:%
\begin{equation*}
\left\{ 
\begin{aligned}
x(t)=&\,\, x_{0}(0)+\int_{0}^{t}A(s)x(s)ds+\int_{0}^{t}b(s,x(s),x_\delta(s))ds
\\
& +\int_{0}^{t}[B(s)x(s)+\sigma (s,x(s),x_{\delta }(s))]dw(s),\,\,t\in
\lbrack 0,T], \\
x(t)=& \,\,x_{0}(t),\,\,t\in \lbrack -\delta ,0].
\end{aligned}%
\right.
\end{equation*}%
or alternatively, for $dt\times dP$-almost all $(t,\omega )\in \lbrack -\delta ,T]\times \Omega $
and all $\varphi \in V$, the following equations hold
\begin{equation}
\left\{ 
\begin{aligned}
\left\langle x(t),\varphi \right\rangle _{H}=&\,\, \left\langle x_{0}(0),\varphi
\right\rangle _{H}+\int_{0}^{t}\left\langle A(s)x(s),\varphi \right\rangle
_{\ast }ds+\int_{0}^{t}\left\langle b(s,x(s),x_{\delta }(s)),\varphi
\right\rangle _{H}ds \\
& +\int_{0}^{t}\langle \lbrack B(s)x(s)+\sigma (s,x(s),
x_{\delta }(s))]dw(s),\varphi \rangle _{H},\,\,t\in \lbrack 0,T], \\
x(t)=&\,\, x_{0}(t),\,\,t\in \lbrack -\delta ,0].
\end{aligned}%
\right.
\end{equation}
\end{definition}

To get existence and uniqueness of the solution to \eqref{sdee},  we impose the following conditions.

\begin{itemize}
\item[(\textbf{A1})] For each $(x,y)\in H\times H,$ $b(\cdot ,\cdot
,x,y)$, $\sigma (\cdot ,\cdot ,x,y)$ are progressively
measurable. $b(\cdot ,\cdot ,0,0)\in L_{\mathbb{F}}^{2}(0,T;H)$, $\sigma
(\cdot ,\cdot ,0,0)\in L_{\mathbb{F}}^{2}(0,T;\mathcal{L}_{2}^{0})$ and $%
x_{0}(\cdot )\in L^{2}([-\delta ,0];V)\cap C([-\delta ,0];H)$, where $%
C([-\delta ,0];H)$ is the space of continuous functions mapping from $[-\delta
,0]$ to $H$.

\item[(\textbf{A2})] For each $x\in V,$ $A(\cdot ,\cdot )x$ and $B(\cdot
,\cdot )x$ are progressively measurable. There exist $\alpha >0$ and $
\lambda \in \mathbb{R}$ such that for almost all $(t,\omega )\in \lbrack 0,T]\times
\Omega $, 
\begin{equation*}
2\left\langle A(t)x,x\right\rangle _{\ast }+\Vert B(t)x\Vert _{\mathcal{L}%
_{2}^{0}}^{2}\leq -\alpha \Vert x\Vert _{V}^{2}+\lambda \Vert x\Vert
_{H}^{2},\,\,\,\text{ for all } x\in V.
\end{equation*}

\item[(\textbf{A3})] There exists a constant $K_{1}>0$ such that for almost all $(t,\omega )\in \lbrack 0,T]\times \Omega $, 
\begin{equation*}
\Vert A(t)x\Vert _{\ast }\leq K_{1}\Vert x\Vert _{V},\,\,\,\text{ for all } x\in V.
\end{equation*}

\item[(\textbf{A4})] There exists a constant $K>0$ such that the following
Lipschitz condition holds: for almost all $(t,\omega )\in \lbrack 0,T]\times \Omega 
$ and all $x,x^{\prime}, y, y'\in H$,
\begin{align*}
& \Vert b(t,x,y)-b(t,x^{\prime },y')\Vert
_{H}^{2}+\Vert \sigma (t,x,y)-\sigma (t,x^{\prime },y')\Vert _{\mathcal{L}_{2}^{0}}^{2} \\
& \leq K(\Vert x-x^{\prime }\Vert _{H}^{2}+\Vert y-y'\Vert _{H}^{2}).
\end{align*}
\end{itemize}

Note that   (A2) and (A3) imply that there exists a constant $C_{1}$ depending on $
\lambda $ and $K$ such that
\begin{equation}
\Vert B(t)x\Vert _{\mathcal{L}_{2}^{0}}\leq C_{1}\Vert x\Vert
_{V},\text{ for all }  x\in V.  \label{Myeq2-41}
\end{equation}

The following lemma will be used frequently later.

\begin{lemma}
\label{Le3-2} Assume $\eta \in L^{2}([-\delta ,t];E)$ and the function $%
\lambda:[-\delta ,t]\rightarrow \mathbb{\R_+}$ is non-increasing. Then we have 
\begin{equation}
\int_{0}^{t}\lambda (s)\Vert \eta _{\delta }(s)\Vert _{E}^{2}ds\leq
m^{2}([-\delta ,0])\int_{-\delta }^{t}\lambda (s)\Vert \eta (s)\Vert
_{E}^{2}ds,  \label{Myeq2-7}
\end{equation}
where $\eta_\delta$ is given in \eqref{Myeq3-3}.
\end{lemma}

\begin{proof}
By the H\"{o}lder's inequality, we have
\begin{eqnarray*}
\int_{0}^{t}\lambda (s)\Vert \eta _{\delta }(s)\Vert _{E}^{2}ds
&=&\int_{0}^{t}\lambda (s)\left\Vert \int_{-\delta }^{0}\eta (s+r)m(dr)\right\Vert
_{E}^{2}ds \\
&\leq &m([-\delta ,0])\int_{0}^{t}\lambda (s)\int_{-\delta }^{0}\Vert \eta
(s+r)\Vert _{E}^{2}m(dr)ds \\
&=&m([-\delta ,0])\int_{-\delta }^{0}\int_{0}^{t}\lambda (s)\Vert \eta
(s+r)\Vert _{E}^{2}dsm(dr) \\
&=&m([-\delta ,0])\int_{-\delta }^{0}\int_{r}^{t+r}\lambda (s-r)\Vert \eta
(s)\Vert _{E}^{2}dsm(dr).
\end{eqnarray*}%
Then from the non-increasing assumption on $\lambda ,$ we have%
\begin{eqnarray*}
\int_{0}^{t}\lambda (s)\Vert \eta _{\delta }(s)\Vert _{E}^{2}ds &\leq
&m([-\delta ,0])\int_{-\delta }^{0}\int_{r}^{t+r}\lambda (s)\Vert \eta
(s)\Vert _{E}^{2}dsm(dr) \\
&\leq &m([-\delta ,0])\int_{-\delta }^{0}\int_{-\delta }^{t}\lambda (s)\Vert
\eta (s)\Vert _{E}^{2}dsm(dr) \\
&=&m^{2}([-\delta ,0])\int_{-\delta }^{t}\lambda (s)\Vert \eta (s)\Vert
_{E}^{2}ds.
\end{eqnarray*}
\end{proof}

We then have the following a priori estimate on the solution of SDEE.
\begin{theorem}
\label{estimate} Let the assumptions (A1)-(A4) hold. Suppose that $x(\cdot )$
is a solution to SDEE~\eqref{sdee} associated with $(x_{0},b,\sigma )$ in the sense of Definition \ref{def:SDEE}.
Then 
\begin{align}
& \mathbb{E}\Big[\sup\limits_{0\leq t\leq T}\Vert x(t)\Vert _{H}^{2}\Big]+%
\mathbb{E}\int_{0}^{T}\Vert x(t)\Vert _{V}^{2}dt  \notag  \label{es-x} \\
& \leq C\bigg\{\Vert x_{0}(0)\Vert _{H}^{2}+\int_{-\delta }^{0}\Vert
x(t)\Vert _{H}^{2}dt+\mathbb{E}\int_{0}^{T}\left(\Vert b(t,0,0)\Vert
_{H}^{2}+\Vert \sigma (t,0,0)\Vert _{\mathcal{L}_{2}^{0}}^{2}\right)dt\bigg\},
\end{align}%
for some constant $C>0$ depending on $\lambda ,\alpha ,K_{1}$ and $K$.
Moreover, if $x^{\prime }(\cdot )$ is a   solution to 
\eqref{sdee} with $(x_{0}^{\prime },b^{\prime },\sigma ^{\prime })$, we have 
    \begin{align}
& \mathbb{E}\Big[\sup\limits_{0\leq t\leq T}\Vert x(t)-x^{\prime }(t)\Vert
_{H}^{2}\Big]+\mathbb{E}\int_{0}^{T}\Vert x(t)-x^{\prime }(t)\Vert
_{V}^{2}dt\notag\\
&\leq  C\Bigg\{\Vert x_{0}(0)-x_{0}^{\prime }(0)\Vert _{H}^{2}  
 +\int_{-\delta }^{0}\Vert x_{0}(t)-x_{0}^{\prime }(t)\Vert
_{H}^{2}ds+\mathbb{E}\int_{0}^{T}\Vert b(t,x^{\prime }(t),x_{\delta
}^{\prime }(t))-b^{\prime }(t,x^{\prime }(t),x_{\delta }^{\prime }(t))\Vert
_{H}^{2}dt\notag \\
& \qquad\qquad +\mathbb{E}\int_{0}^{T}\Vert \sigma (t,x^{\prime
}(t),x_{\delta }^{\prime }(t))-\sigma ^{\prime }(t,x^{\prime }(t),x_{\delta
}^{\prime }(t))\Vert _{\mathcal{L}_{2}^{0}}^{2}dt\Bigg\}. \label{es-difference}
\end{align}
\end{theorem}

\begin{proof}  
We only prove the estimate \eqref{es-difference}, and the estimation \eqref{es-x} follows from \eqref{es-difference} with $x_{0}^{\prime
}(0)=0,$ $b^{\prime }=\sigma ^{\prime }=0$.  To simplify the
notations, we denote 
\begin{equation*}
\hat{x}(t)=x(t)-x^{\prime }(t),\text{ for }t\in \lbrack -\delta ,T].
\end{equation*}%
We first note that, by Lemma \ref{Le3-2},\ for $\tilde{x}=x,x^{\prime },%
\hat{x}$, $\mathbb{E}\int_{0}^{t}\Vert \tilde{x}_{\delta }(s)\Vert _{H}^{2}ds\leq C%
\mathbb{E}\int_{-\delta }^{t}\Vert \tilde{x}(s)\Vert _{H}^{2}ds.$ Thus, noting that $\tilde x\in L_{\mathbb F}^2(0,T;H)$ and by the assumptions on $b$ and $\sigma$, we have $b(\cdot ,\tilde{x}(\cdot ),\tilde{x}_{\delta }(\cdot ))\in L_{\mathbb{F}
}^{2}(0,T;H)$ and $\sigma (\cdot ,\tilde{x}(\cdot ),\tilde{x}_{\delta
}(\cdot ))\in L_{\mathbb{F}}^{2}(0,T;\mathcal{L}_{2}^{0}).$ Applying It\^{o}'s formula \eqref{e:ito-lemma} to $\|\hat x(t)\|^2_H$ on $[0,T]$, we have for $t\in \lbrack 0,T]$ that 
\begin{align*}
& \Vert \hat{x}(t)\Vert _{H}^{2}-\Vert \hat{x}(0)\Vert
_{H}^{2}=2\int_{0}^{t}\left\langle A(s)\hat{x}(s),\hat{x}(s)\right\rangle
_{\ast }ds \\
& \hspace{2em}+2\int_{0}^{t}\left\langle b(s,x(s),x_{\delta }(s))-b^{\prime
}(s,x^{\prime }(s),x_{\delta }^{\prime }(s)),\hat{x}(s)\right\rangle _{H}ds
\\
& \hspace{2em}+2\int_{0}^{t}\left\langle [B(s)\hat{x}(s)+\sigma
(s,x(s),x_{\delta }(s))-\sigma ^{\prime }(s,x^{\prime }(s),x_{\delta
}^{\prime }(s))]dw(s),\hat{x}(s)\right\rangle _{H} \\
& \hspace{2em}+\int_{0}^{t}\Vert B(s)\hat{x}(s)+\sigma (s,x(s),x_{\delta
}(s))-\sigma ^{\prime }(s,x^{\prime }(s),x_{\delta }^{\prime }(s))\Vert _{%
\mathcal{L}_{2}^{0}}^{2}ds.
\end{align*}%
By the assumption (A2), we have 
\begin{align*}
\Vert \hat{x}(t)\Vert _{H}^{2}& \leq \Vert \hat{x}(0)\Vert _{H}^{2}-\alpha
\int_{0}^{t}\Vert \hat{x}(s)\Vert _{V}^{2}ds+(1+\lambda )\int_{0}^{t}\Vert 
\hat{x}(s)\Vert _{H}^{2}ds \\
& \hspace{1em}+\int_{0}^{t}\Vert b(s,x(s),x_{\delta }(s))-b^{\prime
}(s,x^{\prime }(s),x_{\delta }^{\prime }(s))\Vert _{H}^{2}ds \\
& \hspace{1em}+\int_{0}^{t}\Vert \sigma (s,x(s),x_{\delta }(s))-\sigma
^{\prime }(s,x^{\prime }(s),x_{\delta }^{\prime }(s))\Vert _{\mathcal L_2^0}^{2}ds \\
& \hspace{1em}+2\int_{0}^{t}\big<B(s)\hat{x}(s),\sigma (s,x(s),x_{\delta
}(s))-\sigma ^{\prime }(s,x^{\prime }(s),x_{\delta }^{\prime }(s))\big>_{%
\mathcal{L}_{2}^{0}}ds \\
& \hspace{1em}+2\int_{0}^{t}\left\langle [B(s)\hat{x}(s)+\sigma
(s,x(s),x_{\delta }(s))-\sigma ^{\prime }(s,x^{\prime }(s),x_{\delta
}^{\prime }(s))]dw(s),\hat{x}(s)\right\rangle _{H}.
\end{align*}%
Then by applying triangular inequality and the inequality $2ab\le pa^2+b^2/p$ for all $p>0$ together with condition (A4), \eqref{Myeq2-41} and \eqref{Myeq2-7}, we can get 
\begin{equation}
\begin{aligned}
\Vert \hat{x}(t)\Vert _{H}^{2}& \leq \Vert \hat{x}(0)\Vert _{H}^{2}-\alpha
\int_{0}^{t}\Vert \hat{x}(s)\Vert _{V}^{2}ds+2Km^{2}([-\delta ,0])(1+\tfrac{%
2(C_{1})^{2}}{\alpha })\int_{-\delta }^{0}\Vert \hat{x}(s)\Vert _{H}^{2}ds \\
& \hspace{1em}+\Big(1+\lambda +2K(1+m^{2}([-\delta ,0]))(1+\tfrac{%
2(C_{1})^{2}}{\alpha })\Big)\int_{0}^{t}\Vert \hat{x}(s)\Vert _{H}^{2}ds \\
& \hspace{1em}+2\int_{0}^{t}\Vert b(s,x^{\prime }(s),x_{\delta }^{\prime
}(s))-b^{\prime }(s,x^{\prime }(s),x_{\delta }^{\prime }(s))\Vert _{H}^{2}ds
\\
& \hspace{1em}+\Big(2+\frac{4(C_{1})^{2}}{\alpha }\Big)\int_{0}^{t}\Vert \sigma
(s,x^{\prime }(s),x_{\delta }^{\prime }(s))-\sigma ^{\prime }(s,x^{\prime
}(s),x_{\delta }^{\prime }(s))\Vert _{\mathcal{L}_{2}^{0}}^{2}ds+\frac{\alpha }{2}\int_{0}^{t}\Vert \hat{x}(s)\Vert _{V}^{2}ds \\
& \hspace{1em}+2\int_{0}^{t}\left\langle [B(s)\hat{x}(s)+\sigma
(s,x(s),x_{\delta }(s))-\sigma ^{\prime }(s,x^{\prime }(s),x_{\delta
}^{\prime }(s))]dw(s),\hat{x}(s)\right\rangle _{H}.
\end{aligned}
\label{Myeq2-3}
\end{equation}%
Taking expectation on both sides, we have
\begin{align*}
\mathbb{E}[\Vert \hat{x}(t)\Vert _{H}^{2}]+\frac{\alpha }{2}\mathbb{E}
\int_{0}^{t}\Vert \hat{x}(s)\Vert _{V}^{2}ds\leq & \Vert \hat{x}(0)\Vert
_{H}^{2}+C\Bigg\{\int_{-\delta }^{0}\Vert \hat{x}(s)\Vert _{H}^{2}ds+\mathbb{E%
}\int_{0}^{t}\Vert \hat{x}(s)\Vert _{H}^{2}ds \\
& +\mathbb{E}\int_{0}^{t}\Vert b(s,x^{\prime }(s),x_{\delta }^{\prime
}(s))-b^{\prime }(s,x^{\prime }(s),x_{\delta }^{\prime }(s))\Vert _{H}^{2}ds
\\
& +\mathbb{E}\int_{0}^{t}\Vert \sigma (s,x^{\prime }(s),x_{\delta }^{\prime
}(s))-\sigma ^{\prime }(s,x^{\prime }(s),x_{\delta }^{\prime }(s))\Vert _{%
\mathcal{L}_{2}^{0}}^{2}ds\Bigg\}.
\end{align*}
Applying the Gr\"onwall's inequality on $[0,T]$, we obtain 
\begin{equation}
\begin{split}
& \mathbb{E}[\Vert \hat{x}(t)\Vert _{H}^{2}]+\mathbb{E}\int_{0}^{t}\Vert 
\hat{x}(s)\Vert _{V}^{2}ds \\
& \leq C\Bigg\{\Vert \hat{x}(0)\Vert _{H}^{2}+\int_{-\delta }^{0}\Vert \hat{x}%
(s)\Vert _{H}^{2}ds+\mathbb{E}\int_{0}^{t}\Vert b(s,x^{\prime }(s),x_{\delta
}^{\prime }(s))-b^{\prime }(s,x^{\prime }(s),x_{\delta }^{\prime }(s))\Vert
_{H}^{2}ds \\
& \hspace{4em}+\mathbb{E}\int_{0}^{t}\Vert \sigma (s,x^{\prime
}(s),x_{\delta }^{\prime }(s))-\sigma ^{\prime }(s,x^{\prime }(s),x_{\delta
}^{\prime }(s))\Vert _{\mathcal{L}_{2}^{0}}^{2}ds\Bigg\}.
\end{split}
\label{Myeq2-5}
\end{equation}

On the other hand, by \eqref{e:QV} and  Burkholder-Davis-Gundy inequality, we have 
\begin{equation}
\begin{split}
& \mathbb{E}\Big[\sup\limits_{0\leq t\leq T}\int_{0}^{t}\left\langle [B(s)%
\hat{x}(s)+\sigma (s,x(s),x_{\delta }(s))-\sigma ^{\prime }(s,x^{\prime
}(s),x_{\delta }^{\prime }(s))]dw(s),\hat{x}(s)\right\rangle _{H}\Big]
\\
& \leq C\mathbb{E}\left( \int_{0}^{T}\Vert B(t)\hat{x}(t)+\sigma
(t,x(t),x_{\delta }(t))-\sigma ^{\prime }(t,x^{\prime }(t),x_{\delta
}^{\prime }(t))\Vert _{\mathcal{L}_{2}^{0}}^{2}\Vert \hat{x}(t)\Vert
_{H}^{2}dt\right) ^{\frac{1}{2}} \\
& \leq C\mathbb{E}\left( \sup\limits_{0\leq t\leq T}\Vert \hat{x}(t)\Vert
_{H}\Big(\int_{0}^{T}\Vert B(t)\hat{x}(t)+\sigma (t,x(t),x_{\delta
}(t))-\sigma ^{\prime }(t,x^{\prime }(t),x_{\delta }^{\prime }(t))\Vert _{%
\mathcal{L}_{2}^{0}}^{2}dt\Big)^{\frac{1}{2}}\right) \\
& \leq \frac{1}{4}\mathbb{E}\Big[\sup\limits_{0\leq t\leq T}\Vert \hat{x}%
(t)\Vert _{H}^{2}\Big]+C\mathbb{E}\int_{0}^{T}\Vert B(t)\hat{x}(t)+\sigma
(t,x(t),x_{\delta }(t))-\sigma ^{\prime }(t,x^{\prime }(t),x_{\delta
}^{\prime }(t))\Vert _{\mathcal{L}_{2}^{0}}^{2}dt \\
& \leq \frac{1}{4}\mathbb{E}\Big[\sup\limits_{0\leq t\leq T}\Vert \hat{x}%
(t)\Vert _{H}^{2}\Big]+{C}\mathbb{E}\int_{0}^{T}[\Vert B(t)\hat{x}(t)\Vert _{%
\mathcal{L}_{2}^{0}}^{2}+\Vert \sigma (t,x(t),x_{\delta }(t))-\sigma
^{\prime }(t,x^{\prime }(t),x_{\delta }^{\prime }(t))\Vert _{\mathcal{L}%
_{2}^{0}}^{2}]dt.
\end{split}
\label{e:3}
\end{equation}%
By Theorem 3.2 in \cite{79k}, it can be inferred  that $\hat{x}(t)\in C([0,T],H)$ for
almost all $\omega $. Taking supremum over $t\in \lbrack 0,T]$ on both sides
of (\ref{Myeq2-3}), then it follows from (\ref{e:3}), (\ref{Myeq2-41}) and (\ref
{Myeq2-5}) that
\begin{equation}
\begin{split}
& \mathbb{E}\Big[\sup\limits_{0\leq t\leq T}\Vert \hat{x}(t)\Vert _{H}^{2}%
\Big] \\
& \leq C\Big\{\Vert \hat{x}(0)\Vert _{H}^{2}+\int_{-\delta }^{0}\Vert \hat{x}%
(s)\Vert _{H}^{2}ds+\mathbb{E}\int_{0}^{T}\Vert b(s,x^{\prime }(s),x_{\delta
}^{\prime }(s))-b^{\prime }(s,x^{\prime }(s),x_{\delta }^{\prime }(s))\Vert
_{H}^{2}ds \\
& \hspace{4em}+\mathbb{E}\int_{0}^{T}\Vert \sigma (s,x^{\prime
}(s),x_{\delta }^{\prime }(s))-\sigma ^{\prime }(s,x^{\prime }(s),x_{\delta
}^{\prime }(s))\Vert _{\mathcal{L}_{2}^{0}}^{2}ds\Big\}.
\end{split}
\label{Myeq3-6}
\end{equation}
The desired result then follows from (\ref{Myeq2-5}) and (\ref{Myeq3-6}).
\end{proof}

The well-posedness of SDEE \eqref{sdee} is established in the following theorem.
\begin{theorem}
\label{exu} Under conditions (A1)-(A4), equation \eqref{sdee} has a unique
solution $x(\cdot )\in L_{\mathbb{F}}^{2}(-\delta ,T;V)$.
\end{theorem}
\begin{remark}
 Due to the presence of unbounded operator $B$ in the diffusion term in \eqref{sdee},  we cannot apply directly the well-posedness result in \cite{TGR02} as done in \cite{14sms}. Thus, we provide a detailed proof in Appendix \ref{se:Thm3.2} for the existence and uniqueness of solution to  equation \eqref{sdee} by a standard method of employing  the fixed-point theorem. 
\end{remark}

\subsection{Anticipated backward stochastic evolution equations}

In this subsection, we present a well-posedness theory for a type of \emph{anticipated
backward stochastic evolution equations} (ABSEEs). It will be used to
characterize the adjoint processes in the derivation of the maximum
principle.

Let
$\mathcal{M}:[0,T]\times \Omega \rightarrow \mathcal{L}(V,V^{\ast }),\,\,\,%
\mathcal{N}:[0,T]\times \Omega \rightarrow \mathcal{L}(\mathcal{L}%
_{2}^{0},V^{\ast })$
and $g:[0,T]\times \Omega \times V\times \mathcal{L}
_{2}^{0}\times V\times \mathcal{L}_{2}^{0}\to H$ be progressively measurable mappings. Let $F$ be a process on $[0,T]$ with bounded variation. Then we consider the
following ABSEE with a datum acting on an interval (a running terminal)
\begin{equation}\label{absee-0}
\left\{ 
\begin{aligned}
p(t)=&\,\, \xi+\int_{(t,T]}\zeta(s)dF(s)+\int_t^T\Big\{\mathcal{M}(s)p(s)+\mathcal{N}(s)q(s)\\
&+\mathbb{E}^{\mathcal{F}%
_{s}}\big[g(s,p(s),q(s),\int_{-\delta }^{0}p(s-r)m(dr),\int_{-\delta
}^{0}q(s-r)m(dr))\big]\Big\}ds\\
& -\int_t^Tq(s)dw(s),\quad t\in \lbrack 0,T], \\
p(t)=&\,\, 0,\,q(t)=0,\quad t\in (T,T+\delta ],
\end{aligned}%
\right. 
\end{equation}%
where $\xi$ and $\zeta$ are the terminal conditions acting on $T$ and $(0,T]$ respectively, and $m(\cdot )$ is a
finite measure on $[-\delta ,0]$. We remark that in general $dF(s)$ induces a signed measure on $[0,T]$, and in the special case when $dF(s)$ is absolutely continuous with respect to the Lebesgue measure,  we can put the term $\int_{t}^{T}\zeta(s)dF(s)$ into the generator and the ABSEE \eqref{absee-0} reduces to a standard form without datum acting on an interval. 

Similar to the notation $\eta_\delta(t)$ introduced in \eqref{Myeq3-3},  we set for $t\in \lbrack 0,T]$,  
\begin{equation}\label{e:eta-delta+}
\eta _{\delta _{+}}(t):=\int_{-\delta }^{0}\eta (t-s)m(ds),
\end{equation}
whenever the integration of the function $\eta (t-\cdot ):[-\delta
,0]\rightarrow E$ with respect to $m$ exists. Then the integrals with respect to $m(ds)$ appearing in \eqref{absee-0} can be abbreviated as
\begin{equation*}
p_{\delta _{+}}(t)=\int_{-\delta }^{0}p(t-s)m(ds)\text{ and }q_{\delta
_{+}}(t)=\int_{-\delta }^{0}q(t-s)m(ds).
\end{equation*}

 Recall that  $L_{\mathbb{F},F}^{2}(0,T;H)$  denotes the set of all $H$-valued progressively measurable processes $\phi$ satisfying
\begin{equation*}
\mathbb{E}\Big[ \int_{0}^{T}\Vert \phi (t)\Vert _{H}^{2}d|F|_v(t)\Big] <\infty.
\end{equation*}

To obtain the existence and uniqueness of solutions to \eqref{absee-0}, we
impose the following assumptions.

\begin{itemize}
\item[(\textbf{B1})] For each $(p,q,p_{\delta _{+}},q_{\delta _{+}}),$ $%
g(\cdot ,\cdot ,p,q,p_{\delta _{+}},q_{\delta _{+}})$ is a measurable function. $%
g(\cdot ,\cdot ,0,0,0,0)\in L_{\mathbb{F}}^{1,2}(0,T;H)$, $\xi \in L^{2}(%
\mathcal{F}_{T};H)$ and $\zeta\in L_{\mathbb{F},F}^{2}(0,T;H)$, with $L_{\mathbb{F}}^{1,2}(0,T;H)$ being the
space of $H$-valued progressively measurable processes $\phi(\cdot)$ with  norm
$$\Vert \phi\Vert_{L_{\mathbb{F}}^{1,2}(0,T;H)}=\left(\mathbb{\mathbb{E}}
\left[\left(\int_{0}^{T}\Vert \phi(t)\Vert_{H}{d}t\right)^{2}\right]\right)^{\frac{1}{2}}.$$ 

\item[(\textbf{B2})] For each $x\in V,$ $\mathcal{M}(\cdot ,\cdot )x$ and $%
\mathcal{N}(\cdot ,\cdot )x$ are progressively measurable. There exist
constants $\alpha >0$ and $\lambda \in \mathbb{R}$ such that for each $%
(t,\omega )\in \lbrack 0,T]\times \Omega $, 
\begin{equation*}
2\left\langle \mathcal{M}(t)x,x\right\rangle _{\ast }+\Vert \mathcal{N}
^{\ast }(t)x\Vert _{\mathcal{L}_{2}^{0}}^{2}\leq -\alpha \Vert x\Vert
_{V}^{2}+\lambda \Vert x\Vert _{H}^{2},\quad \text{ for all } x\in V,
\end{equation*}
where $\mathcal N^*\in \mathcal L(V, \mathcal L_2^0)$ is the adjoint operator of $\mathcal N\in \mathcal L(\mathcal L_2^0, V^*)$. 

\item[(\textbf{B3})] There exists a constant $K_{1}>0$ such that for almost all $(t,\omega )\in \lbrack 0,T]\times \Omega $, 
\begin{equation*}
\Vert \mathcal{M}(t)x\Vert _{\ast }\leq K_{1}\Vert x\Vert
_{V},\,\,\,\forall x\in V.
\end{equation*}

\item[(\textbf{B4})] There exists a positive constant $K$ such that for
almost all $(t,\omega )\in \lbrack 0,T]\times \Omega $, 
\begin{align*}
& \Vert g(t,p,q,p_{\delta _{+}},q_{\delta _{+}})-g(p^{\prime },q^{\prime
},p_{\delta _{+}}^{\prime },q_{\delta _{+}}^{\prime })\Vert _{H}^{2} \\
& \leq K\Big(\Vert p-p^{\prime }\Vert _{V}^{2}+\Vert q-q^{\prime }\Vert _{%
\mathcal{L}_{2}^{0}}^{2}+\Vert p_{\delta _{+}}-p_{\delta _{+}}^{\prime
}\Vert _{V}^{2}+\Vert q_{\delta _{+}}-q_{\delta _{+}}^{\prime }\Vert _{%
\mathcal{L}_{2}^{0}}^{2}\Big)
\end{align*}%
holds for $(p,q,p_{\delta _{+}},q_{\delta _{+}}),(p^{\prime },q^{\prime
},p_{\delta _{+}}^{\prime },q_{\delta _{+}}^{\prime })\in V\times \mathcal{L}%
_{2}^{0}\times V\times \mathcal{L}_{2}^{0}$;

\item[(\textbf{B5})] The total variation $|F|_v$ of $F$ is uniformly bounded by $K_F$. 
\end{itemize}

Similar to \eqref{Myeq2-41},  (B2) and (B3) yield
\begin{equation}
\Vert \mathcal{N}(t)x\Vert_{V^\ast}\leq C_{1}\Vert x\Vert
_{\mathcal{L}_{2}^{0}},\text{ for }x\in \mathcal{L}_{2}^{0},  \label{Myeq2-41(2)}
\end{equation}
 where $C_{1}$ is a constant depending on $\lambda $ and $K$.
\begin{remark}
If $\mathcal{M}$ and $\mathcal{N}$ are the adjoint operators of $A$
and $B$ respectively which satisfy the conditions (A2)-(A3), $\mathcal{M}$ and $\mathcal{N}$ satisfy (B2)-(B3) naturally.
\end{remark}

\begin{definition}\label{def:ABSEE}
We say the process $(p(\cdot ),q(\cdot ))\in L_{\mathbb{F}}^{2}(0,T+\delta
;V)\times L_{\mathbb{F}}^{2}(0,T+\delta ;\mathcal{L}_{2}^{0})$ is a
solution to \eqref{absee-0}, if for  $dt\times dP$-almost all $(t,\omega )\in \lbrack 0,T+\delta
]\times \Omega ,$ it holds in $V^{\ast }$ that:%
\begin{equation*}
\left\{ 
\begin{aligned}
p(t)=&\,\, \xi +\int_{(t,T]}\zeta(s)dF(s)+\int_{t}^{T}\bigg\{\mathcal{M}(s)p(s)+\mathcal{N}(s)q(s)\\
&+\mathbb{%
E}^{\mathcal{F}_{s}}[g(s,p(s),q(s),p_{\delta _{+}}(s),q_{\delta
_{+}}(s))]\bigg\}ds  -\int_{t}^{T}q(s)dw(s),\,\,t\in \lbrack 0,T], \\
p(t)=&\,\, 0,\ q(t)=0,\,\,t\in (T,T+\delta ],
\end{aligned}%
\right.
\end{equation*}%
or alternatively, for   $dt\times dP$-almost all $(t,\omega )\in \lbrack 0,T]\times \Omega $ and
every $\varphi \in V$, 
\begin{equation}
\left\{ 
\begin{aligned}
\left\langle p(t),\varphi \right\rangle _{H}=&\,\, \left\langle \xi ,\varphi
\right\rangle _{H}+\int_{(t,T]}\langle\zeta(s),\varphi\rangle_H dF(s)+\int_{t}^{T}\left\langle \mathcal{M}(s)p(s),\varphi
\right\rangle _{\ast }ds\\
& +\int_{t}^{T}\left\langle \mathcal{N}(s)q(s),\varphi
\right\rangle _{\ast }ds +\int_{t}^{T}\big<\mathbb{E}^{\mathcal{F}_{s}}[g(s,p(s),q(s),p_{\delta
_{+}}(s),q_{\delta _{+}}(s))],\varphi \big>_{H}ds \\
& -\int_{t}^{T}\left\langle q(s)dw(s),\varphi \right\rangle
_{H},\,\,t\in \lbrack 0,T], \\
p(t)=&\,\, 0,\ q(t)=0,\,\,t\in (T,T+\delta ].
\end{aligned}%
\right.
\end{equation}
\end{definition}

In parallel with Lemma \ref{Le3-2}, we have the following result. 

\begin{lemma}
\label{Le3-3} For a process $\zeta \in L^{2}([t,T+\delta ];E)$ and a
non-decreasing function $\lambda:[t,T+\delta ]\rightarrow \R_+$, we
have 
\begin{equation}
\int_{t}^{T}\lambda (s)\Vert \zeta _{\delta _{+}}(s)\Vert _{E}^{2}ds\leq
m^{2}([-\delta ,0])\int_{t}^{T+\delta }\lambda (s)\Vert \zeta (s)\Vert
_{E}^{2}ds.  \label{Myeq3-7}
\end{equation}
\end{lemma}

\begin{proof}
From the H\"{o}lder's inequality and the non-decreasing property of $\lambda
,$ 
\begin{align*}
\int_{t}^{T}\lambda (s)\Vert \zeta _{\delta _{+}}(s)\Vert _{E}^{2}ds &\leq
m([-\delta ,0])\int_{t}^{T}\lambda (s)\int_{-\delta }^{0}\Vert \zeta
(s-r)\Vert _{E}^{2}m(dr)ds \\
&=m([-\delta ,0])\int_{-\delta }^{0}\int_{t}^{T}\lambda (s)\Vert \zeta
(s-r)\Vert _{E}^{2}dsm(dr) \\
&=m([-\delta ,0])\int_{-\delta }^{0}\int_{t-r}^{T-r}\lambda (s+r)\Vert \zeta
(s)\Vert _{E}^{2}dsm(dr) \\
&\leq m([-\delta ,0])\int_{-\delta }^{0}\int_{t-r}^{T-r}\lambda (s)\Vert
\zeta (s)\Vert _{E}^{2}dsm(dr) \\
&\leq m([-\delta ,0])\int_{-\delta }^{0}\int_{t}^{T+\delta }\lambda (s)\Vert
\zeta (s)\Vert _{E}^{2}dsm(dr) \\
&=m^{2}([-\delta ,0])\int_{t}^{T+\delta }\lambda (s)\Vert \zeta (s)\Vert
_{E}^{2}ds.
\end{align*}

This completes the proof.
\end{proof}

In parallel to the a apriori estimate (Theorem \ref{estimate}) for SDEE \eqref{sdee},  we have the following result for ABSEE~(\ref{absee-0}). 

\begin{theorem}
\label{es-p} Assume the assumptions (B1)-(B4) hold. If $(p(\cdot),q(\cdot))$ is a solution to ABSEE~(\ref{absee-0})
associated with $(\xi,g,\zeta)$ in the sense of Definition \ref{def:ABSEE}, there exists a positive constants $C$
depending on $\lambda,\alpha,K$ and $K_{1}$ such that
\begin{align*}
&  \mathbb{E}\big[\sup\limits_{0\leq t\leq T}\Vert p(t)\Vert_{H}%
^{2}\big]+\mathbb{E}\int_{0}^{T}\Vert q(t)\Vert_{\mathcal{L}_{2}^{0}}%
^{2}dt+\mathbb{E}\int_{0}^{T}\Vert p(t)\Vert_{V}^{2}dt\\
&  \leq C\bigg\{\mathbb{E}[\Vert\xi\Vert_{H}^{2}]+\mathbb{E}%
\int_{(0,T]}\Vert\zeta(t)\Vert_{H}^{2}\Delta F(t)dF(t)+\mathbb{E}%
\Big(\int_{(0,T]}\Vert\zeta(t)\Vert_{H}d|F|_v(t)\bigg)^{2}\\
&  \hspace{2em}+\Big(\mathbb{E}\int_{0}^{T}\Vert g(t,0,0,0,0)\Vert_{H}dt\Big)^{2}\Big\}.
\end{align*}
Moreover, let $(p^{\prime}(\cdot),q^{\prime}(\cdot))$ be a solution to \eqref{absee-0} with $(\xi^{\prime},g^{\prime
},\zeta^{\prime})$. Then the following estimate holds:
\begin{align}
&  \mathbb{E}\big[\sup\limits_{0\leq t\leq T}\Vert p(t)-p^{\prime}(t)\Vert
_{H}^{2}\big]+\mathbb{E}\int_{0}^{T}\Vert q(t)-q^{\prime}(t)\Vert
_{\mathcal{L}_{2}^{0}}^{2}dt+\mathbb{E}\int_{0}^{T}\Vert p(t)-p^{\prime
}(t)\Vert_{V}^{2}dt\leq C\bigg\{\mathbb{E}[\Vert\xi-\xi^{\prime}\Vert_{H}%
^{2}]\nonumber\label{bes-difference}\\
&  \hspace{2em}+\mathbb{E}\int_{(0,T]}\Vert\zeta(t)-\zeta^{\prime}(t)\Vert_{H}^{2}\Delta
F(t)dF(t)+\mathbb{E}\Big(\int_{(0,T]}\Vert\zeta(t)-\zeta^{\prime}(t)\Vert_{H}%
d|F|_v(t)\Big)^{2}\\
&  \hspace{2em}+\mathbb{E}\Big(\int_{0}^{T}\Vert g(t,p^{\prime}(t),q^{\prime}%
(t),p_{\delta_{+}}^{\prime}(t),q_{\delta_{+}}^{\prime}(t))-g^{\prime
}(t,p^{\prime}(t),q^{\prime}(t),p_{\delta_{+}}^{\prime}(t),q_{\delta_{+}%
}^{\prime}(t))\Vert_{H}dt\Big)^{2}\bigg\},\nonumber
\end{align}
where $C$ is a positive constant depending on $\lambda,\alpha,K$ and $K_{1}.$
\end{theorem}

\begin{proof}
It suffices to prove \eqref{bes-difference}, which implies the first one. Set
\[
\hat{p}(t)=p(t)-p^{\prime}(t),\ \hat{q}(t)=q(t)-q^{\prime}(t),\ \hat{\zeta
}(t)=\zeta(t)-\zeta^{\prime}(t), \text{ and }\hat{\xi}=\xi-\xi^{\prime}.
\]
Applying It\^{o}'s formula \eqref{e:ito-lemma} to $\Vert\hat{p}(t)\Vert_{H}^{2}$ on $[t,T],$ we have
\begin{align*}
&  \Vert\hat{p}(t)\Vert_{H}^{2}+\int_{t}^{T}\Vert\hat{q}(s)\Vert
_{\mathcal{L}_{2}^{0}}^{2}ds=\Vert\hat{\xi}\Vert_{H}^{2}+2\int_{t}%
^{T}\Big\{\big<\mathcal{M}(s)\hat{p}(s),\hat{p}(s)\big>_{\ast}+\left\langle
\mathcal{N}(s)\hat{q}(s),\hat{p}(s)\right\rangle _{\ast}\\
&  +\big<\mathbb{E}^{\mathcal{F}_{s}}[g(s,p(s),q(s),p_{\delta_{+}%
}(s),q_{\delta_{+}}(s))-g^{\prime}(s,p^{\prime}(s),q^{\prime}(s),p_{\delta
_{+}}^{\prime}(s),q_{\delta_{+}}^{\prime}(s))],\hat{p}(s)\big>_{H}\Big\}ds\\
&  +2\int_{(t,T]}\big<\hat{p}(s),\hat{\zeta}(s)\big>_{H}dF(s)+\int%
_{(t,T]}\Vert\hat{\zeta}(s)\Vert_{H}^{2}\Delta F(s)dF(s)\\
&  -2\int_{t}^{T}\left\langle \hat{p}(s),\hat{q}(s)dw(s)\right\rangle _{H}.
\end{align*}

By assumptions (B2)-(B3), we obtain that, for some positive constant
$\varepsilon$ to be determined,
\begin{align}
&  \Vert\hat{p}(t)\Vert_{H}^{2}+\int_{t}^{T}\Vert\hat{q}(s)\Vert
_{\mathcal{L}_{2}^{0}}^{2}ds\nonumber\label{EQ-5}\\
&  \leq\Vert\hat{\xi}\Vert_{H}^{2}+2\int_{(t,T]}\big<\hat{p}(s),\hat{\zeta
}(s)\big>_{H}dF(s)+\int_{(t,T]}\Vert\hat{\zeta}(s)\Vert_{H}^{2}\Delta
F(s)dF(s)\nonumber\\
&  \hspace{1em}+\int_{t}^{T}\Big\{-2\varepsilon\big<\mathcal{M}(s)\hat
{p}(s),\hat{p}(s)\big>_{\ast}+2(1+\varepsilon)\big<\mathcal{M}(s)\hat
{p}(s),\hat{p}(s)\big>_{\ast}+(1+\varepsilon)\Vert\mathcal{N}^{\ast}(s)\hat
{p}(s)\Vert^{2}_{\mathcal L_2^0}+\tfrac{1}{1+\varepsilon}\Vert\hat{q}(s)\Vert_{\mathcal{L}%
_{2}^{0}}^{2}\nonumber\\
&  \hspace{2.5em}+2\big|\big<\mathbb{E}^{\mathcal{F}_{s}}%
[g(s,p(s),q(s),p_{\delta_{+}}(s),q_{\delta_{+}}(s))-g(s,p^{\prime
}(s),q^{\prime}(s),p_{\delta_{+}}^{\prime}(s),q_{\delta_{+}}^{\prime
}(s))],\hat{p}(s)\big>_{H}\big|\nonumber\\
&  \hspace{2.5em}+2\big|\big<\mathbb{E}^{\mathcal{F}_{s}}[g(s,p^{\prime
}(s),q^{\prime}(s),p_{\delta_{+}}^{\prime}(s),q_{\delta_{+}}^{\prime
}(s))-g^{\prime}(s,p^{\prime}(s),q^{\prime}(s),p_{\delta_{+}}^{\prime
}(s),q_{\delta_{+}}^{\prime}(s))],\hat{p}(s)\big>_{H}\big|\Big\}ds\nonumber\\
&  \hspace{1em}-2\int_{t}^{T}\left\langle \hat{q}(s)dw(s),\hat{p}%
(s)\right\rangle _{H}\\
&  \leq\Vert\hat{\xi}\Vert_{H}^{2}+2\sup_{t\leq s\leq T}\Vert\hat{p}%
(s)\Vert_{H}\int_{(t,T]}\Vert\hat{\zeta}(s)\Vert_{H}d|F|_v(s)+\int_{(t,T]}%
\Vert\hat{\zeta}(s)\Vert_{H}^{2}\Delta F(s)dF(s)\nonumber\\
&  \hspace{1em}+\int_{t}^{T}\Big\{2\varepsilon K_{1}\Vert\hat{p}(s)\Vert
_{V}^{2}+(1+\varepsilon)(-\alpha\Vert\hat{p}(s)\Vert_{V}^{2}+\lambda\Vert
\hat{p}(s)\Vert_{H}^{2})+\tfrac{1}{1+\varepsilon}\Vert\hat{q}(s)\Vert
_{\mathcal{L}_{2}^{0}}^{2}\nonumber\\
&  \hspace{2.5em}+\tfrac{4K(1+m^{2}([-\delta,0]))}{\varepsilon}\Vert\hat
{p}(s)\Vert_{H}^{2}+\tfrac{\varepsilon}{4K(1+m^{2}([-\delta,0]))}
\Big\Vert\mathbb{E}^{\mathcal{F}_{s}}[g(s,p(s),q(s),p_{\delta_{+}}(s),q_{\delta
_{+}}(s))\nonumber\\
&  \hspace{12em}-g(s,p^{\prime}(s),q^{\prime}(s),p_{\delta_{+}}^{\prime
}(s),q_{\delta_{+}}^{\prime}(s))]\Big\Vert_{H}^{2}\Big\}ds\nonumber\\
&  \hspace{1em}+2\sup_{t\leq s\leq T}\Vert\hat{p}(s)\Vert_{H}\int_{t}^{T}\Big
\Vert\mathbb{E}^{\mathcal{F}_{s}}[g(s,p^{\prime}(s),q^{\prime}(s),p_{\delta
_{+}}^{\prime}(s),q_{\delta_{+}}^{\prime}(s))\nonumber\\
&  \hspace{4em}-g^{\prime}(s,p^{\prime}(s),q^{\prime}(s),p_{\delta_{+}%
}^{\prime}(s),q_{\delta_{+}}^{\prime}(s))]\Big\Vert_{H}ds-2\int_{t}^{T}%
\left\langle \hat{q}(s)dw(s),\hat{p}(s)\right\rangle _{H},\nonumber
\end{align}
Taking expectation on both sides and combining the Lipschitz condition (B4)
with Lemma \ref{Le3-3} (noting $\hat p(t)=0$ and $\hat q(t)=0$ for $t\in(T, T+\delta]$), we deduce that
\begin{align*}
&  \mathbb{E}[\Vert\hat{p}(t)\Vert_{H}^{2}]+\mathbb{E}\int_{t}^{T}\Vert\hat
{q}(s)\Vert_{\mathcal{L}_{2}^{0}}^{2}ds\\
&  \leq\mathbb{E}[\Vert\hat{\xi}\Vert_{H}^{2}]+2\mathbb{E}\big[\sup_{t\leq
s\leq T}\Vert\hat{p}(s)\Vert_{H}\int_{(t,T]}\Vert\hat{\zeta}(s)\Vert
_{H}d|F|_v(s)\big]+\mathbb{E}\int_{(t,T]}\Vert\hat{\zeta}(s)\Vert_{H}^{2}\Delta
F(s)dF(s)\\
&  \hspace{1em}+\mathbb{E}\int_{t}^{T}\Big\{2\varepsilon K_{1}\Vert\hat
{p}(s)\Vert_{V}^{2}+(1+\varepsilon)(-\alpha\Vert\hat{p}(s)\Vert_{V}%
^{2}+\lambda\Vert\hat{p}(s)\Vert_{H}^{2})+\tfrac{1}{1+\varepsilon}\Vert\hat
{q}(s)\Vert_{\mathcal{L}_{2}^{0}}^{2}\\
&  \hspace{4em}+\tfrac{4K(1+m^{2}([-\delta,0]))}{\varepsilon}\Vert\hat
{p}(s)\Vert_{H}^{2}+\tfrac{\varepsilon}{4(1+m^{2}([-\delta,0]))}\big\{\Vert
\hat{p}(s)\Vert_{V}^{2}+\Vert\hat{q}(s)\Vert_{\mathcal{L}_{2}^{0}}^{2}\\
&  \hspace{15em}+ \Vert\hat{p}_{\delta_{+}}(s)\Vert_{V}^{2}+\Vert\hat{q}_{\delta_{+}}(s)\Vert_{\mathcal{L}_{2}^{0}}^{2}\big\}\Big\}ds\\
&  \hspace{1em}+2\mathbb{E}\Big[\sup_{t\leq s\leq T}\Vert\hat{p}(s)\Vert
_{H}\int_{t}^{T}\Big\Vert\mathbb{E}^{\mathcal{F}_{s}}[g(s,p^{\prime}(s),q^{\prime
}(s),p_{\delta_{+}}^{\prime}(s),q_{\delta_{+}}^{\prime}(s))\\
&  \hspace{14em}-g^{\prime}(s,p^{\prime}(s),q^{\prime}(s),p_{\delta_{+}%
}^{\prime}(s),q_{\delta_{+}}^{\prime}(s))]\Big\Vert_{H}ds\Big]\\
&  \leq\mathbb{E}[\Vert\hat{\xi}\Vert_{H}^{2}]+2\mathbb{E}\big[\sup_{t\leq
s\leq T}\Vert\hat{p}(s)\Vert_{H}\int_{(t,T]}\Vert\hat{\zeta}(s)\Vert
_{H}d|F|_v(s)\big]+\mathbb{E}\int_{(t,T]}\Vert\hat{\zeta}(s)\Vert_{H}^{2}\Delta
F(s)dF(s)\\
&  \hspace{1em}+\mathbb{E}\int_{t}^{T}\Big\{\big(2\varepsilon K_{1}%
+\tfrac{\varepsilon}{4}-\alpha(1+\varepsilon)\big)\Vert\hat{p}(s)\Vert_{V}%
^{2}+\big(\tfrac{1}{1+\varepsilon}+\tfrac{\varepsilon}{4}\big)\Vert\hat
{q}(s)\Vert_{\mathcal{L}_{2}^{0}}^{2}\\
&  \hspace{6em}+\big(\tfrac{4K(1+m^{2}([-\delta,0]))}{\varepsilon}%
+\lambda(1+\varepsilon)\big)\Vert\hat{p}(s)\Vert_{H}^{2}\Big\}ds\\
&  \hspace{1em}+2\mathbb{E}\Big[\sup_{t\leq s\leq T}\Vert\hat{p}(s)\Vert
_{H}\int_{t}^{T}\Vert\mathbb{E}^{\mathcal{F}_{s}}[g(s,p^{\prime}(s),q^{\prime
}(s),p_{\delta_{+}}^{\prime}(s),q_{\delta_{+}}^{\prime}(s))\\
&  \hspace{14em}-g^{\prime}(s,p^{\prime}(s),q^{\prime}(s),p_{\delta_{+}%
}^{\prime}(s),q_{\delta_{+}}^{\prime}(s))]\Vert_{H}ds\Big].
\end{align*}
Choosing $\varepsilon$ small enough such that
\[
2\varepsilon K_{1}+\frac{\varepsilon}{4}-\alpha(1+\varepsilon)<0\ \ \text{and
}\ \ \frac{1}{1+\varepsilon}+\frac{\varepsilon}{4}<1\Longleftrightarrow \varepsilon\in(0,3),
\]
we can get
\begin{align*}
&  \mathbb{E}[\Vert\hat{p}(t)\Vert_{H}^{2}]+\mathbb{E}\int_{t}^{T}\Vert\hat
{q}(s)\Vert_{\mathcal{L}_{2}^{0}}^{2}ds+\mathbb{E}\int_{t}^{T}\Vert\hat
{p}(s)\Vert_{V}^{2}ds\\
&  \leq C\mathbb{E}\bigg\{\int_{t}^{T}\Vert\hat{p}(s)\Vert_{H}^{2}ds+\Vert
\hat{\xi}\Vert_{H}^{2}+\sup_{t\leq s\leq T}\Vert\hat{p}(s)\Vert_{H}%
\int_{(t,T]}\Vert\hat{\zeta}(s)\Vert_{H}d|F|_v(s)+\int_{(t,T]}\Vert\hat{\zeta
}(s)\Vert_{H}^{2}\Delta F(s)dF(s)\\
&  \hspace{3em}+\sup_{0\leq s\leq T}\Vert\hat{p}(s)\Vert_{H}\int_{0}^{T}%
\Big\Vert\mathbb{E}^{\mathcal{F}_{s}}[g(s,p^{\prime}(s),q^{\prime}(s),p_{\delta
_{+}}^{\prime}(s),q_{\delta_{+}}^{\prime}(s))\\
&  \hspace{14em}-g^{\prime}(s,p^{\prime}(s),q^{\prime}(s),p_{\delta_{+}%
}^{\prime}(s),q_{\delta_{+}}^{\prime}(s))]\Big\Vert_{H}ds\bigg\},
\end{align*}
where $C$ is a positive constant depending on $\lambda,\alpha,K,K_{1}$.
Applying  Gr\"onwall's inequality to $\mathbb{E}[\Vert\hat{p}(t)\Vert_{H}^{2}]$ yields that, for some undetermined $a>0,$
\begin{align}
&  \mathbb{E}[\Vert\hat{p}(t)\Vert_{H}^{2}]+\mathbb{E}\int_{t}^{T}\Vert\hat
{q}(s)\Vert_{\mathcal{L}_{2}^{0}}^{2}ds+\mathbb{E}\int_{t}^{T}\Vert\hat
{p}(s)\Vert_{V}^{2}ds\nonumber\label{supout}\\
&  \leq C\mathbb{E}\Big\{\Vert\hat{\xi}\Vert_{H}^{2}+\sup_{t\leq s\leq T}%
\Vert\hat{p}(s)\Vert_{H}\int_{(t,T]}\Vert\hat{\zeta}(s)\Vert_{H}
d|F|_v(s)+\int_{(t,T]}\Vert\hat{\zeta}(s)\Vert_{H}^{2}\Delta
F(s)dF(s)\nonumber\\
&  \hspace{3em}+\sup_{0\leq s\leq T}\Vert\hat{p}(s)\Vert_{H}\int_{0}^{T}\Big
\Vert\mathbb{E}^{\mathcal{F}_{s}}[g(s,p^{\prime}(s),q^{\prime}(s),p_{\delta
_{+}}^{\prime}(s),q_{\delta_{+}}^{\prime}(s))\nonumber\\
&  \hspace{14em}-g^{\prime}(s,p^{\prime}(s),q^{\prime}(s),p_{\delta_{+}%
}^{\prime}(s),q_{\delta_{+}}^{\prime}(s))]\Big\Vert_{H}ds\Big\}\\
&  \leq C_{1}\mathbb{E}\Big\{\Vert\hat{\xi}\Vert_{H}^{2}+\int_{(0,T]}\Vert
\hat{\zeta}(t)\Vert_{H}^{2}\Delta F(t)dF(t)+\frac{1}{a}\Big(\int_{(0,T]}\Vert\hat{\zeta}(t)\Vert_{H}d|F|_v(t)\Big)^{2}+a\sup_{0\leq t\leq
T}\Vert\hat{p}(t)\Vert_{H}^{2}\nonumber\\
&  \hspace{1em}+\frac{1}{a}\Big(\int_{0}^{T}\Vert g(t,p^{\prime
}(t),q^{\prime}(t),p_{\delta_{+}}^{\prime}(t),q_{\delta_{+}}^{\prime
}(t))-g^{\prime}(t,p^{\prime}(t),q^{\prime}(t),p_{\delta_{+}}^{\prime
}(t),q_{\delta_{+}}^{\prime}(t))\Vert_{H}dt\Big)^{2}\Big\},\nonumber
\end{align}
with $C_{1}$ being a constant independent of $a$ and may vary from line
to line.

On the other hand, by \eqref{e:QV} and Burkholder-Davis-Gundy inequality, we have for some positive constant $D$, 
\begin{equation}%
\begin{split}
\mathbb{E}\Big[\sup\limits_{0\leq t\leq T}\Big\vert\int_{t}^{T}\left\langle
\hat{q}(s)dw(s),\hat{p}(s)\right\rangle _{H}\Big\vert\Big] 
&  \leq D \mathbb{E}\Big(\int_{0}^{T}\Vert\hat{q}(t)\Vert_{\mathcal{L}_{2}^{0}%
}^{2}\Vert\hat{p}(t)\Vert_{H}^{2}dt\Big)^{\frac{1}{2}}\\
&  \leq D \mathbb{E}\Big[\sup\limits_{0\leq t\leq T}\Vert\hat{p}(t)\Vert
_{H}\Big(\int_{0}^{T}\Vert\hat{q}(t)\Vert_{\mathcal{L}_{2}^{0}}^{2}%
dt\Big)^{\frac{1}{2}}\Big]\\
&  \leq\frac{1}{8}\mathbb{E}\big[\sup\limits_{0\leq t\leq T}\Vert\hat
{p}(t)\Vert_{H}^{2}\big]+2 D^2 \mathbb{E}\int_{0}^{T}\Vert\hat{q}(t)\Vert
_{\mathcal{L}_{2}^{0}}^{2}dt.
\end{split}
\label{EQ-6}%
\end{equation}
Then taking supremum over $t\in\lbrack0,T]$ on both sides of (\ref{EQ-5}) (for
any fixed $\varepsilon>0$), we get
\begin{align*}
&  \sup\limits_{0\leq t\leq T}\Vert\hat{p}(t)\Vert_{H}^{2}\leq C_{1}%
\Big\{\Vert\hat{\xi}\Vert_{H}^{2}+\Big(\int_{(0,T]}\Vert\hat{\zeta}%
(s)\Vert_{H}d|F|_v(s)\Big)^{2}+\int_{(0,T]}\Vert\hat{\zeta}(s)\Vert_{H}%
^{2}\Delta F(s)dF(s)\\
&  \hspace{1em}+\int_{0}^{T}\Vert\hat{p}(s)\Vert_{H}^{2}ds+\int_{0}^{T}%
\Vert\hat{q}(s)\Vert_{\mathcal{L}_{2}^{0}}^{2}ds+\int_{0}^{T}\Vert\hat
{p}(s)\Vert_{V}^{2}ds+\int_{0}^{T}\mathbb{E}^{\mathcal{F}_{s}}[\Vert\hat
{p}_{\delta_{+}}(s)\Vert_{V}^{2}+\Vert\hat{q}_{\delta_{+}}(s)\Vert
_{\mathcal{L}_{2}^{0}}^{2}]ds\\
&  \hspace{1em}+\Big(\int_{0}^{T}\Vert\mathbb{E}^{\mathcal{F}_{s}%
}[g(s,p^{\prime}(s),q^{\prime}(s),p_{\delta_{+}}^{\prime}(s),q_{\delta_{+}%
}^{\prime}(s))-g^{\prime}(s,p^{\prime}(s),q^{\prime}(s),p_{\delta_{+}}%
^{\prime}(s),q_{\delta_{+}}^{\prime}(s))]\Vert_{H}ds\Big)^{2}\Big\}\\
&  \hspace{1em}+\frac{1}{4}\sup_{0\leq s\leq T}\Vert\hat{p}(s)\Vert_{H}%
^{2}+2\sup\limits_{0\leq t\leq T}\Big|\int_{t}^{T}\left\langle \hat{p}%
(s),\hat{q}(s)dw(s)\right\rangle _{H}\Big|.
\end{align*}
Taking expectation on both sides, we then obtain by (\ref{EQ-6}), Lemma
\ref{Le3-3} and (\ref{supout}) that
\begin{align}
&  \mathbb{E}\Big[\sup\limits_{0\leq t\leq T}\Vert\hat{p}(t)\Vert_{H}
^{2}\Big]\leq(\frac{1}{2}+C_{1}a)\mathbb{E}\big[\sup\limits_{0\leq t\leq
T}\Vert\hat{p}(t)\Vert_{H}^{2}\big]\nonumber\label{supin}\\
&  +C_{1}\Big\{\mathbb{E}\Vert\hat{\xi}\Vert_{H}^{2}+\int_{(0,T]}\Vert
\hat{\zeta}(S)\Vert_{H}^{2}\Delta F(s)dF(s)+(1+\frac{1}{a})\Big(\int%
_{(0,T]}\Vert\hat{\zeta}(s)\Vert_{H}d|F|_v(s)\Big)^{2}\\
&  \hspace{0.5em}+(1+\frac{1}{a})\Big(\int_{0}^{T}\Vert g(s,p^{\prime
}(s),q^{\prime}(s),p_{\delta_{+}}^{\prime}(s),q_{\delta_{+}}^{\prime
}(s))-g^{\prime}(s,p^{\prime}(s),q^{\prime}(s),p_{\delta_{+}}^{\prime
}(s),q_{\delta_{+}}^{\prime}(s))\Vert_{H}ds\Big)^{2}\Big\}.\nonumber
\end{align}
Choosing sufficiently small $a$, we get%
\begin{align}
&  \mathbb{E}\Big[\sup\limits_{0\leq t\leq T}\Vert\hat{p}(t)\Vert_{H}%
^{2}\Big]\nonumber\\
&  \leq C_1\Big\{\mathbb{E}\Vert\hat{\xi}\Vert_{H}^{2}+\int_{(0,T]}\Vert
\hat{\zeta}(s)\Vert_{H}^{2}\Delta F(s)dF(s)+\Big(\int_{(t,T]}\Vert\hat{\zeta
}(s)\Vert_{H}d|F|_v(s)\Big)^{2}\\
&  \hspace{0.5em}+\Big(\int_{0}^{T}\Vert g(s,p^{\prime}(s),q^{\prime
}(s),p_{\delta_{+}}^{\prime}(s),q_{\delta_{+}}^{\prime}(s))-g^{\prime
}(s,p^{\prime}(s),q^{\prime}(s),p_{\delta_{+}}^{\prime}(s),q_{\delta_{+}%
}^{\prime}(s))]\Vert_{H}\Big\}ds\Big)^{2}\Big\}.\nonumber
\end{align}
This together with (\ref{supout}) yields the desired estimate \eqref{bes-difference} and 
the proof is concluded.
\end{proof}

Now we are ready to prove the well-posedness result for  ABSEE (\ref{absee-0}).
\begin{theorem}\label{ABSEE-thm}
Assuming (B1)-(B5), there exists  a unique solution $(p(\cdot ),q(\cdot ))\in L_{\mathbb{F}}^{2}(0,T+\delta
;V)\times L_{\mathbb{F}}^{2}(0,T+\delta ; \mathcal{L}_{2}^{0})$ of ABSEE (\ref{absee-0})
in the sense of Definition \ref{def:ABSEE}.
\end{theorem}
\begin{proof}
The uniqueness follows directly from (\ref{bes-difference}) in Theorem \ref{es-p}. The proof of the existence  is divided into the following three steps.

\textit{Step 1. The case of $\zeta\equiv 0$.} We shall make use of the continuation method as in \cite{pw99}. For any $\mu \in
\lbrack 0,1]$ and $f_{0}(\cdot )\in L_{\mathbb{F}}^{1,2}(0,T;H)$, we consider
the ABSEE 
\begin{equation} \label{absee2}
\left\{ 
\begin{aligned}
-dp(t)=&\,\, \Big\{\mathcal{M}(t)p(t)+\mathcal{N}(t)q(t)+\mu \mathbb{E}^{%
\mathcal{F}_{t}}\big[g(t,p(t),q(t),p_{\delta _{+}}(t),q_{\delta _{+}}(t))%
\big]+f_{0}(t)\Big\}dt \\
& \quad -q(t)dw(t),\quad t\in \lbrack 0,T], \\
p(T)=&\,\, \xi ,\ p(t)=0,\ q(t)=0,\quad t\in (T,T+\delta ].
\end{aligned}%
\right. 
\end{equation}%
In the following, we shall prove the well-posedness of the above equation,
which implies the desired result by setting $\mu =1$ and $f_{0}(\cdot
)=0.$

For $\mu =0$, by using a similar approach as in \cite[Proposition 3.2]{10dm}, we can show that
\eqref{absee2} has a unique solution for any $f_{0}(\cdot )\in L_{\mathbb{F}}^{1,2}(0,T;H)$.  This well-posedness result can be extended to all $\mu
\in \lbrack 0,1]$ as follows.

Suppose that equation \eqref{absee2} admits a
unique solution for all $f_{0}(\cdot )\in L_{\mathbb{F}}^{1,2}(0,T;H)$ and some fixed $\mu _{0}\in[0,1)$. Then, for an arbitrarily fixed $f_{0}(\cdot )\in L_{\mathbb{F}}^{1,2}(0,T;H)$,  any
given $(P(\cdot ),Q(\cdot ))\in L_{\mathbb{F}}^{2}(0,T+\delta ;V)\times L_{\mathbb{F}}^{2}(0,T+\delta ;\mathcal {L}_{2}^{0})$, and some $\mu \in \lbrack 0,1]$ to be determined,   the following ABSEE 
\begin{equation}\label{absee3}
\left\{ 
\begin{aligned}
-dp(t)=&\,\, \Big\{\mathcal{M}(t)p(t)+\mathcal{N}(t)q(t)+\mu _{0}\mathbb{E}^{\mathcal{F}_{t}}\big[g(t,p(t),q(t),p_{\delta _{+}}(t),q_{\delta _{+}}(t))%
\big]+f_{0}(t) \\
& \qquad +(\mu -\mu _{0})\mathbb{E}^{\mathcal{F}_{t}}\big[g(t,P(t),Q(t),P_{\delta
_{+}}(t),Q_{\delta _{+}}(t))\big]\Big\}dt \\
& \quad -q(t)dw(t),\,\,t\in \lbrack 0,T], \\
p(T)=& \,\,\xi,\ p(t)=0,\,\,q(t)=0,\,\,t\in (T,T+\delta ],
\end{aligned}
\right.  
\end{equation}
admits a unique solution $(p(\cdot ),q(\cdot ))\in L_{\mathbb{F}
}^{2}(0,T+\delta ;V)\times L_{\mathbb{F}}^{2}(0,T+\delta ;\mathcal{L}%
_{2}^{0})$. by this, we can define the solution mapping $%
I:L_{\mathbb F}^{2}(0,T+\delta ;V)\times L_{\mathbb F}^{2}(0,T+\delta ;\mathcal{L}_{2}^{0})\to L_{\mathbb F}^{2}(0,T+\delta ;V)\times L_{\mathbb F}^{2}(0,T+\delta ;\mathcal{L}_{2}^{0})$ by
\begin{equation*}
(P,Q)\mapsto I(P,Q):=(p,q).
\end{equation*}%
Given $(P_{1}(\cdot ),Q_{1}(\cdot )),(P_{2}(\cdot ),Q_{2}(\cdot ))\in L_{\mathbb{F}}^{2}(0,T+\delta ;V)\times L_{\mathbb{F}}^{2}(0,T+\delta ;\mathcal{L}_{2}^{0})$, by Theorem \ref{es-p} we have 
\begin{align*}
& \mathbb{E}\Big[\int_{0}^{T}\Vert p_{1}(t)-p_{2}(t)\Vert
_{V}^{2}dt+\int_{0}^{T}\Vert q_{1}(t)-q_{2}(t)\Vert _{\mathcal{L}%
_{2}^{0}}^{2}dt\Big] \\
& \leq C|\mu -\mu _{0}|^{2}\mathbb{E}\Big[\int_{0}^{T}\Vert
P_{1}(t)-P_{2}(t)\Vert _{V}^{2}dt+\int_{0}^{T}\Vert Q_{1}(t)-Q_{2}(t)\Vert _{%
\mathcal{L}_{2}^{0}}^{2}dt\Big],
\end{align*}%
where $C$ is a positive constant independent of $\mu$. Thus,  for $\mu \in
\lbrack \mu _{0}-\frac{1}{\sqrt{2C}},\mu _{0}+\frac{1}{\sqrt{2C}}]$, the solution 
mapping $I$ is a contraction mapping on $L_{\mathbb{F}}^{2}(0,T;V)\times L_{\mathbb{F}}^{2}(0,T;\mathcal{L}_{2}^{0})$, which implies the well-posedness of (\ref
{absee3}). So starting with $\mu_0=0$ and repeating the above
procedure, we can prove that there exists a unique solution to (\ref{absee2}) for all $\mu\in \lbrack 0,1]$.

\textit{Step 2. The case of $\zeta$ taking values in $V$.} Denote
\[\alpha(t)=\int_{(0,t]}\zeta(s)dF(s),\ t\in [0,T]\]
and
\[\bar p(t)=p(t)+\alpha(t),\ t\in [0,T]\ \text{and} \ \bar p(t)=0,\ t\in (T,T+\delta].\] Then we can rewrite \eqref{absee-0} as
\[
\left\{
\begin{aligned} \bar{p}(t)=&\,\, \xi+ \alpha(T)+\int_t^T\Big\{\mathcal{M}(s)\bar{p}(s)-\mathcal{M}(s)\alpha(s)+\mathcal{N}(s)q(s)\\ &+\mathbb{E}^{\mathcal{F}_{s}}\big[g(s,\bar{p}(s)-\alpha(s),q(s),\bar{p}_{\delta _{+}}(s)-\alpha_{\delta _{+}}(s),q_{\delta _{+}}(s))\big]\Big\}ds \\ & -\int_t^Tq(s)dw(s),\,\,t\in \lbrack 0,T], \\ \bar{p}(t)=&0,\,q(t)=0,\,\,t\in (T,T+\delta ], \end{aligned}\right.
\]
By Step 1,  we know that the above equation admits a unique  solution $(\bar{p}(\cdot),q(\cdot))\in L_{\mathbb{F}
}^{2}(0,T+\delta;V)\times L_{\mathbb{F}}^{2}(0,T+\delta;\mathcal{L}_{2}^{0}
)$. Then it is easy to check that $(\bar{p}(\cdot)-\alpha(\cdot),q(\cdot))\in
L_{\mathbb{F}}^{2}(0,T+\delta;V)\times L_{\mathbb{F}}^{2}(0,T+\delta
;\mathcal{L}_{2}^{0})$ is a unique solution to (\ref{absee-0}).

\textit{Step 3. The case of $\zeta$ taking values in $H$.} Consider the following approximation equations, for $n\ge 1$,
\begin{equation}
\left\{
\begin{aligned} p^n(t)=&\,\, \xi+\int_{(t,T]}\zeta^n(s)dF(s)+\int_t^T\Big\{\mathcal{M}(s)p^n(s)+\mathcal{N}(s)q^n(s)\\ &+\mathbb{E}^{\mathcal{F}_{s}}\big[g(s,p^n(s),q^n(s),p^n_{\delta _{+}}(s),q^n_{\delta _{+}}(s))\big]\Big\}ds\\ & -\int_t^Tq^n(s)dw(s),\,\,t\in \lbrack 0,T], \\ p^n(t)=&\,\, 0,\,q^n(t)=0,\,\,t\in (T,T+\delta ], \end{aligned}\right.
\label{absee-01}%
\end{equation}
where $\zeta^{n}$ converges to $\zeta$ in $L_{\mathbb{F},F}^{2}(0,T;H)$ as $n$ goes to infinity. By
Step 2, for each $n$, ABSEE (\ref{absee-01}) has a unique solution $(p^{n},q^{n})\in L_{\mathbb{F}}^{2}(0,T+\delta;V)\times L_{\mathbb{F}}^{2}(0,T+\delta;\mathcal{L}_{2}^{0})$.
According to (\ref{bes-difference}) in Theorem~\ref{es-p}, we have
\begin{align*}
&  \mathbb{E}\big[\sup\limits_{0\leq t\leq T}\Vert p^{n}(t)-p^{m}(t)\Vert
_{H}^{2}\big]+\mathbb{E}\int_{0}^{T}\Vert q^{n}(t)-q^{m}(t)\Vert
_{\mathcal{L}_{2}^{0}}^{2}dt+\mathbb{E}\int_{0}^{T}\Vert p^{n}(t)-p^{m}%
(t)\Vert_{V}^{2}dt\\
&  \leq C\bigg\{\mathbb{E}\int_{(0,T]}\Vert\zeta^{n}(s)-\zeta^{m}(s)\Vert
_{ H}^{2}\Delta F(s)dF(s)+\mathbb{E}\Big(\int_{(0,T]}\Vert\zeta^{n}%
(s)-\zeta^{m}(s)\Vert_{H}d|F|(s)\Big)^{2}\bigg\}\\
&  \leq CK_{F}\mathbb{E}\int_{(0,T]}\Vert\zeta^{n}(t)-\zeta^{m}(t)\Vert
_{H}^{2}d|F|_v(s),
\end{align*}
where the constant $K_F$ is from Assumption (B5). Hence, $p^{n}$ is a Cauchy sequence in both $L_{\mathbb{F}}^{2}(0,T;V)$ and
$D_{\mathbb{F}}^{2}(0,T;H)$ with limit denoted by $p$, and $q^{n}$ is Cauchy sequence in
$L_{\mathbb{F}}^{2}(0,T;\mathcal{L}_{2}^{0})$ with limit denoted by $q.$

Finally,  we deduce that  $(p,q)$ satisfies (\ref{absee-0}) by combining the following estimates:  for each $t\in [0,T]$, as $n\to \infty$, 
\begin{align*}
&\mathbb{E}\Vert p^{n}(t)-p(t)\Vert_{H}^{2}   \rightarrow0,\\
&\mathbb{E}\Big\Vert\int_{(t,T]}\zeta^{n}(s)dF(s)-\int_{(t,T]}\zeta(s)dF(s)\Big\Vert
_{H}^{2}   \leq K_{F}\mathbb{E}\int_{(t,T]}\Vert\zeta^{n}(s)-\zeta
(s)\Vert_{H}^{2}d|F|_v(s)\rightarrow0,\\
&\mathbb{E}\Big\Vert\int_{t}^{T}(\mathcal{M}(s)p^{n}(s)-\mathcal{M}(s)p(s))ds\Big\Vert
_{V^{\ast}}^{2}\notag\\
&\leq T\mathbb{E}\int_{t}^{T}\Vert\mathcal{M}(s)p^{n}(s)-\mathcal{M}(s)p(s)\Vert_{V^{\ast}}^{2}ds \leq TK_{1}\mathbb{E}\int_{t}
^{T}\Vert p^{n}(s)-p(s)\Vert_{V}^{2}ds\rightarrow0,\\
 &\mathbb{E}\Big\Vert\int_{t}^{T}g(s,p^{n}(s),q^{n}(s),p_{\delta_{+}}^{n}(s),q_{\delta_{+}}^{n}(s))ds-\int_{t}^{T}g(s,p(s),q(s),p_{\delta_{+}}(s),q_{\delta_{+}}(s))ds\Big\Vert_{H}^{2} \notag\\
&\leq T\mathbb{E}\int_{t}^{T}\Vert
g(s,p^{n}(s),q^{n}(s),p_{\delta_{+}}^{n}(s),q_{\delta_{+}}^{n}%
(s))-g(s,p(s),q(s),p_{\delta_{+}}(s),q_{\delta_{+}}(s))\Vert_{H}^{2}ds\notag\\
& \leq C\mathbb{E}\int_{t}^{T}\Big[\Vert p^{n}(s)-p(s)\Vert_{V}^{2}+\Vert q^{n}(s)-q(s)\Vert_{\mathcal{L}_{2}^{0}}^{2}\Big]ds\rightarrow0,\\
&\mathbb{E}\Big\Vert\int_{t}^{T}(\mathcal{N}(s)q^{n}(s)-\mathcal{N}(s)q(s))ds\Big\Vert
_{V^{\ast}}^{2} \notag \\
&\leq T\mathbb{E}\int_{t}^{T}\Vert\mathcal{N}(s)q^{n}
(s)-\mathcal{N}(s)q(s)\Vert_{V^{\ast}}^{2}ds \leq C\mathbb{E}\int_{t}^{T}\Vert
q^{n}(s)-q(s)\Vert_{\mathcal{L}_{2}^{0}}^{2}ds\rightarrow0,\\
& \mathbb{E}\Big\Vert\int_{t}^{T}q^{n}(s)dw(s)-\int_{t}^{T}q(s)dw(s)\Big\Vert_{H}%
^{2}=\mathbb{E}\int_{t}^{T}\Vert q^{n}(s)-q(s)\Vert_{\mathcal{L}_{2}^{0}}^{2}ds\rightarrow0.
\end{align*}
The proof is concluded. 
 \end{proof}

\begin{remark}
Considering the ABSEE \eqref{absee-0}, when $H=V=\mathbb{R}^n$, $m$ is a finite \emph{regular} measure, $dF(s)$ induces a finite \emph{regular} measure, and the function $g$ is linear, the equation \eqref{absee-0} reduces to the ABSDE studied in \cite[Theorem 2.4]{guatteri2021stochastic} where the well-posedness was obtained. 
\end{remark}

{

\begin{remark}
We present some possible extensions of the result by adapting   the proof of Theorem~\ref{ABSEE-thm} properly.

a). The process  $\alpha(t):=\int_{(0,t]}\zeta(s)dF(s)$ can be extended to a  semimartingales, i.e., $\alpha(t):=\int_{(0,t]}\zeta(s)dF(s)+M(t)$ where $\zeta$ anb $F$ satisfy the same conditions as in Theorem \ref{ABSEE-thm}, and $M(t)$ is an $H$-valued square-integrable martingale (not necessarily continuous).

b). When adding a new term $\zeta$ to the  ABSEE in Step 2 of the proof, we first consider the case of $V$-valued process $\zeta$, as  the operator $\mathcal{M}$ is defined only on $V$.  On the other hand, instead of extending $\zeta$ from the space $V$ to the space $H$ as done in Step 3,   we may consider  the ABSEE with a general $V$-valued process  $\alpha\in L_{\mathbb{F}}^{1,2}(0,T;V)$ with $\alpha(T)\in L^{2}(\mathcal{F}_T;H)$ in Step 2. Similarly, if we assume the domain of $\mathcal{M}$ is $H$, we may consider  the equation  with a general $H$-valued process $\alpha\in L_{\mathbb{F}}^{1,2}(0,T;H)$  with $\alpha(T)\in L^{2}(\mathcal{F}_T;H)$ in Step 2. 
   

    c). In the results of a priori estimation  (Theorem \ref{es-p}) and well-posedness (Theorem \ref{ABSEE-thm}), the solution $(p,q)$ of ABSEE can be nonzero after $T,$ i.e., $p(t)=\tilde \xi(t),q(t)=\eta(t)$, for $t>T,$ for some $\tilde \xi,\eta\in L_{\mathcal{F}}^{2}(T,T+\delta;H)$ as in \cite{10py}. 
\end{remark}

\section{Stochastic maximum principle}\label{sec:SMP}

In this section, we study  recursive optimal control problems for a class of infinite dimensional delayed systems and derive the Pontryagin's stochastic maximum principle.

\subsection{Formulation of the control problem}

Suppose that the control domain $U$ is a convex subset of a real separable
Hilbert space $H_{1}$. We identify $H_{1}$ with its dual space.  Consider
the following controlled SEE with delay 
\begin{equation}\label{state}
\left\{ 
\begin{aligned}
dx(t)=& \,\,[A(t)x(t)+b(t,x(t),x_{\delta }(t),u(t),u_{\delta }(t))]dt \\
& +[B(t)x(t)+\sigma (t,x(t),x_{\delta }(t),u(t),u_{\delta }(t))]dw(t), \\
x(t)=& \,\,x_{0}(t),\,\,u(t)=u_0(t),\,\,\,t\in \lbrack -\delta ,0].
\end{aligned}%
\right.  
\end{equation}%
Here, $x_{0}(\cdot )\in L^{2}([-\delta ,0];V)\cap C([-\delta ,0];H)$ and $u_0\in L^{2}([-\delta ,0];U)$ are given initial paths, 
\begin{equation*}
(A,B):[0,T]\times \Omega \rightarrow L(V;V^{\ast}\times \mathcal{L}_{2}^{0})
\end{equation*}
are unbounded linear operators, 
\begin{equation*}
(b,\sigma ):[0,T]\times \Omega \times H\times H\times U\times
H_{1}\rightarrow H\times H
\end{equation*}%
are nonlinear functions, and for a given finite measure $m$ on $[-\delta ,0]$,
recall that according to (\ref{Myeq3-3}), 
\begin{equation}\label{e-delay}
x_{\delta }(t)=\int_{-\delta }^{0}x(t+s)m(ds),\,\,u_{\delta
}(t)=\int_{-\delta }^{0}u(t+s)m(ds).
\end{equation}

\begin{remark}
\label{Exa-1} We list some examples of delay in the literature that can be described by \eqref{e-delay}. 

\begin{itemize}
\item[(i)] The \emph{pointwise delay} (see \cite{10cw})  $x_{\delta }(t)=x(t-\delta )\,,$\ $u_{\delta }(t)=u(t-\delta )$ is given by \eqref{e-delay} when  $m(ds)$ is the Dirac measure at $\delta$. 

\item[(ii)]  Let $m(ds)=e^{\lambda s}ds$, where $
\lambda \in \mathbb{R}$ is the averaging parameter. Then 
\begin{equation*}
x_{\delta }(t)=\int_{-\delta }^{0}x(t+s)e^{\lambda s}ds,\text{ }u_{\delta
}(t)=\int_{-\delta }^{0}u(t+s)e^{\lambda s}ds
\end{equation*}
provides the \emph{moving average delay} (see \cite{11osz}). 

\item[(iii)] The \emph{multipoint-wise delay}  is given by \eqref{e-delay}  with the measure $m(ds)$ defined by 
\begin{equation*}
m(A)=\sum_{k=1}^{n}\alpha _{k}m_{k}(A),\text{\quad for } A\in \mathcal{B}([-\delta ,0]),
\end{equation*}
where  $\alpha _{k}$ are positive  constants and  $m_{k}(ds)$ are Dirac measures at $\delta _{k}\in \lbrack -\delta ,0], 1\leq k\leq n$.

\item[(iv)]  The delay measure of the form $m(ds)=f(s)ds$ for some integrable function $f\ge 0$
on $[-\delta ,0]$. 
\end{itemize}
\end{remark}

The cost functional is defined by 
\begin{equation*}
J(u(\cdot )):=y(0),
\end{equation*}
where  $(y,z)$ solves the following  ABSDE
\begin{equation}  \label{y}
\left\{ 
\begin{aligned}
-dy(t)=&\,\,\mathbb{E}^{\mathcal{F}_{t}} \Big[f(t,x(t),x_{\delta }(t),y(t),y_{\delta
_{+}}(t),z(t),z_{\delta _{+}}(t),u(t),u_{\delta }(t))\Big]dt \\
& \quad -z(t)dw(t),\,\,t\in \lbrack 0,T], \\
y(T) =&\,\,h\left(x(T),x_{\delta,\nu}(T)\right),\text{ }y(t)=0,\text{ }z(t)=0,\text{ }t\in (T,T+\delta ].
\end{aligned}%
\right.
\end{equation}
In \eqref{y}, $y_{\delta_+(t)}$ and $z_{\delta_+}(t)$ are defined by \eqref{e:eta-delta+},  the coefficient functions
\begin{equation*}
h:\Omega \times H\times H\rightarrow \mathbb{R}\text{ and }f:[0,T]\times \Omega
\times H\times H\times \mathbb{R}\times \mathbb{R}\times \mathcal L_2^0(K,H)\times \mathcal L_2^0(K,H)\times U\times
H_{1}\rightarrow \mathbb{R},
\end{equation*}
 and  $x_{\delta, \nu}(t)$ is defined in a similar way as in \eqref{Myeq3-3} and \eqref{e:eta-delta+}: 
\begin{equation}\label{e:x=dv}
x_{\delta,\nu}(t):=\int_{-\delta}^{0}x(t+s)\nu(ds),
\end{equation}
where $\nu$ is a finite measure on $[-\delta ,0]$.

The admissible control set $\mathcal U$ is defined by, for some given $u_0\in L^{2}([-\delta ,0];U)$,  
\begin{equation*}
\mathcal{U}:=\Big\{u:[-\delta ,T]\times \Omega \rightarrow U\ \text{satisfying }%
u\in L_{\mathbb{F}}^{2}([0,T];U)\ \text{and}\ u(t)=u_0(t),\,t\in \lbrack
-\delta ,0]\Big\}.
\end{equation*}
We aim to find conditions (i.e., the maximum principle) that are necessary to obtain an optimal control $\bar u$, i.e., an admissible control $\bar{u}(\cdot)$ that minimizes the cost functional $J(u(\cdot ))$ over $\mathcal{U}$.

Assume the following conditions are satisfied.

\begin{itemize}
\item[(\textbf{H1})] $b(\cdot ,\cdot ,0,0,0,0)\in L_{\mathbb{F}}^{2}(0,T;H)$%
, $\sigma (\cdot ,\cdot ,0,0,0,0)\in L_{\mathbb{F}}^{2}(0,T;\mathcal{L}%
_{2}^{0}).$

\item[(\textbf{H2})] The operators $A$ and $B$ satisfy the coercivity and
boundedness conditions (A2)-(A3).

\item[(\textbf{H3})] For each $(x,x_{\delta },u,u_{\delta }),$ the functions 
$b(\cdot ,\cdot ,x,x_{\delta },u,u_{\delta })$ and $\sigma (\cdot ,\cdot
,x,x_{\delta },u,u_{\delta })$ are progressively measurable. For almost all $%
(t,\omega )\in \lbrack 0,T]\times \Omega $, $b$ and $\sigma $ are Fr\'{e}%
chet differentiable with respect to $(x,x_{\delta },u,u_{\delta })$ with
continuous and uniformly bounded derivatives.

\item[(\textbf{H4})] For each $(x,x_{\delta },y,y_{\delta +},z,z_{\delta
+},u,u_{\delta }),$ $f(\cdot ,\cdot ,x,x_{\delta },y,y_{\delta
+},z,z_{\delta +},u,u_{\delta })$ is progressively measurable and $h(\cdot
,x,x_{\delta,\nu})$ is $\mathcal{F}_{T}$-measurable.$\ f$ and $h$ are Fr\'{e}chet
differentiable with continuous derivatives. Moreover, $f_{y},f_{y_{\delta_+}},f_{z},f_{z_{\delta_+}}$ are
bounded and there exists a constant $C$ such that  
\begin{align*}
& \Vert f_{x}(t,x,x_{\delta },y,y_{\delta
+},z,z_{\delta
+},u,u_{\delta })\Vert _{H}+\Vert
f_{x_{\delta }}(t,x,x_{\delta },y,y_{\delta
+},z,y_{\delta
+},u,u_{\delta })\Vert _{H} \\
& +\Vert f_{u}(t,x,x_{\delta },y,y_{\delta
+},z,z_{\delta
+},u,u_{\delta })\Vert _{H_{1}}+\Vert
f_{u_{\delta }}(t,x,x_{\delta },y,y_{\delta
+},z,y_{\delta
+},u,u_{\delta })\Vert _{H_{1}} \\
& \leq C\big(1+\Vert x\Vert _{H}+\Vert x_{\delta }\Vert _{H}+\Vert u\Vert
_{H_{1}}+\Vert u_{\delta }\Vert _{H_{1}}\big)
\end{align*}%
and 
\begin{equation*}
\Vert h_{x}(x,x_{\delta,\nu})\Vert _{H}+h_{x_{\delta,\nu}}(x,x_{\delta,\nu})\Vert _{H}\leq C(1+\Vert x\Vert _{H}+\Vert x_{\delta,\nu}\Vert _{H}).
\end{equation*}
hold for almost all $(t,\omega )\in
\lbrack 0,T]\times \Omega $, $(x,x_{\delta },y,y_{\delta_+},z,z_{\delta_+},u,u_{\delta })\in H\times H\times \R\times\R\times \mathcal L_2^0(K;\R)\times \mathcal L_2^0(K;\R)\times H_1\times H_1$, and $(x,x_{\delta,\nu})\in H\times H$.
\end{itemize}
 The Hamiltonian  $H:\lbrack
0,T]\times \Omega \times H\times H\times \mathbb{R}\times \mathbb{R}\times\mathcal L_2^0(K;\R)\times\mathcal L_2^0(K;\R)\times
U\times H_{1}\times H\times \mathcal{L}_{2}^{0}\times \mathbb{R}\to \R$ is defined as follows 
\begin{equation}\label{Hamiltionian}
\begin{aligned}
 H(t,x,x_{\delta },y,y_{\delta
+},z,z_{\delta
+},u,u_{\delta },p,q,k):=&\left\langle b(t,x,x_{\delta
},u,u_{\delta }),p\right\rangle _{H}+\left\langle \sigma (t,x,x_{\delta
},u,u_{\delta }),q\right\rangle _{\mathcal{L}_{2}^{0}}\\
& -f(t,x,x_{\delta },y,y_{\delta
+},z,z_{\delta
+},u,u_{\delta })k, 
\end{aligned}
\end{equation}
where $k(\cdot )$ is the solution of the adjoint SDDE
associated with $y(\cdot )$:
\begin{equation}\label{k}
\begin{cases} 
dk(t)= [f_{y}(t)k(t)+(f_{y_{_{\delta +}}}k)_{\delta
}(t)]dt+[f_{z}(t)k(t)+(f_{z_{\delta _{+}}}k)_{\delta
}(t)]dw(t),\,\,t\in \lbrack 0,T], \\
k(0)= -1.
\end{cases}
\end{equation}
In \eqref{k}, recalling  \eqref{Myeq3-3} we have
\begin{align*}
(f_{y_{_{\delta +}}}k)_{\delta
}(t)=\int_{-\delta}^0f_{y_{\delta+}}(t+s)k(t+s)m(ds),\, 
(f_{z_{_{\delta +}}}k)_{\delta
}(t)=\int_{-\delta}^0f_{z_{\delta+}}(t+s)k(t+s)m(ds).
\end{align*}
\begin{remark}
    The SDDE (\ref{k}) is not of a standard form, but its well-posedness can be derived according to the same argument as the one in the proof of Theorem \ref{exu}.
\end{remark}

Assume that $\bar{u}(\cdot )\in \mathcal{U}$ is an optimal control, i.e., $J(%
\bar{u}(\cdot ))=\inf\limits_{u(\cdot )\in \mathcal{U}}J(u(\cdot ))$, and $(%
\bar{x}(\cdot ),\bar{y}(\cdot ),\bar{z}(\cdot ))$ are the corresponding
solutions of the system. For the sake of conciseness, we take the following notations:  for $\varphi =b,\sigma $ and $\tau =x,x_{\delta },y,y_{\delta_+},z,z_{\delta_+},u,u_{\delta }$,  
\begin{align*}
&\varphi (t) =\varphi (t,\bar{x}(t),\bar{x}_{\delta }(t),\bar{u}(t),\bar{u}%
_{\delta }(t)), \\
&\varphi _{\tau }(t) =\varphi _{\tau }(t,\bar{x}(t),\bar{x}_{\delta }(t),%
\bar{u}(t),\bar{u}_{\delta }(t)), \\
&f(t) =f(t,\bar{x}(t),\bar{x}_{\delta }(t),\bar{y}(t),\bar{z}(t),\bar{u}(t),%
\bar{u}_{\delta }(t)), \\
&f_{\tau }(t)=f_{\tau }(t,\bar{x}(t),\bar{x}_{\delta }(t),\bar{y}(t),\bar{y}_{\delta+}(t),\bar{z}%
(t),\bar{z}_{\delta+}(t),\bar{u}(t),\bar{u}_{\delta }(t)),\\
&h(T) =h(\bar{x}(T),\bar{x}_{\delta,\nu}(T)), \\
&h_{x }(T) =h_{x }(\bar{x}(T),\bar{x}_{\delta,\nu}(T)),\\
&h_{x_{\delta,\nu} }(T) =h_{x_{\delta,\nu}}(\bar{x}(T),\bar{x}_{\delta,\nu}(T)).
\end{align*}
To proceed with the dual analysis for the control system described by \eqref{state} and \eqref{y}, we introduce the following ABSEE with a datum acting on $[T-\delta,T]$ as the adjoint equation:  
\begin{equation}\label{p}
\left\{ 
\begin{aligned}
p(t)=& -h_{x}(T)k(T)-\int_{I_t}\mathbb{E}^{\mathcal{F}_{s}}%
\big[k(T) h_{x_{\delta,\nu}}(T)\big]\nu(d(s-T))\\
&+\int_t^T\Bigg\{A^{\ast }(s)p(s)+B^{\ast }(s)q(s)+b_{x}^{\ast }(s)p(s)+\sigma
_{x}^{\ast }(s)q(s)-f_{x}(s)k(s) \\
& \hspace{2em}+\mathbb{E}^{\mathcal{F}_{s}}\bigg[\int_{-\delta }^{0}b_{x_{\delta }}^{\ast
}(s-r)p(s-r)m(dr)+\int_{-\delta }^{0}\sigma _{x_{\delta }}^{\ast
}(s-r)q(s-r)m(dr) \\
& \hspace{5.5em}{-}\int_{-\delta }^{0}f_{x_{\delta }}(s-r)k(s-r)m(dr)\bigg]
\Bigg\}ds-\int_t^Tq(s)dw(s),\,\,t\in \lbrack 0,T], \\
p(t)=& 0,\,q(t)=0,\,\,t\in (T,T+\delta ],
\end{aligned}%
\right.  
\end{equation}
where $I_t:=(t,T]\cap [T-\delta,T]$.
According to Theorem \ref{ABSEE-thm}, this equation admits a unique solution $
(p(\cdot ),q(\cdot ))\in L_{\mathbb{F}}^{2}(0,T+\delta ;V)\times L_{\mathbb{F}}^{2}(0,T+\delta ;\mathcal{L}_{2}^{0})$ by taking $\zeta(t)=-\mathbb{E}^{\mathcal{F}_{t}}[k(T) h_{x_{\delta,\nu}}(T)]$ and $dF(t)=\nu(d(t-T)) I_{[T-\delta,T]}(t)$ (or equivalently $\zeta(t)=\mathbb{E}^{\mathcal{F}_{t}}[k(T) h_{x_{\delta,\nu}}(T)]$ and $dF(t)=-\nu(d(t-T)) I_{[T-\delta,T]}(t)$).

 Denoting, for $\tau =x,x_{\delta },y,y_{\delta+},z,z_{\delta+},u,u_{\delta }$, 
 \begin{equation}\label{e:H-H-tau}
\begin{aligned}
H(t)& =H(t,\bar{x}(t),\bar{x}_{\delta }(t),\bar{y}(t),\bar{y}_{\delta+}(t),\bar{z}(t),\bar{z}_{\delta+}(t),\bar{u}(t),
\bar{u}_{\delta }(t),p(t),q(t),k(t)), \\
H_{\tau }(t)& =H_{\tau }(t,\bar{x}(t),\bar{x}_{\delta }(t),\bar{y}(t),\bar{y}_{\delta+}(t),\bar{z}%
(t),\bar{z}_{\delta+}(t),\bar{u}(t),\bar{u}_{\delta }(t),p(t),q(t),k(t)),
\end{aligned}
\end{equation}
 the adjoint equation \eqref{p} can be written as
\begin{equation} \label{ade-1}
\left\{ 
\begin{aligned}
p(t)=& \,\,-h_{x}(\bar{x}(T))k(T)-\int_{I_t}\mathbb{E}^{\mathcal{F}_{s}}
\big[k(T) h_{x_{\delta,\nu}}(T)\big]\nu(d(s-T))\\
&\quad  +\int_t^T\bigg\{A^{\ast }(s)p(s)+B^{\ast }(s)q(s)+H_{x}(s)+\mathbb{E}^{\mathcal{F}_{s}}\Big[\int_{-\delta }^{0}H_{x_{\delta }}(s-r)m(dr)\Big]\bigg\}
ds\\&\qquad  -\int_t^Tq(s)dw(s),\,\,t\in \lbrack 0,T], \\
p(t)=& \,\,0,\,q(t)=0,\,\,t\in (T,T+\delta ].
\end{aligned}
\right. 
\end{equation}%

\subsection{Variational equations}

Let $\bar u(\cdot)\in \mathcal U$ be an optimal control and $(\bar x(\cdot), \bar y(\cdot), \bar z(\cdot))$ be the corresponding solutions.  For $u(\cdot )\in \mathcal{U}$, $\rho \in \lbrack 0,1]$, define the 
perturbation of $\bar u(\cdot)$ by
\begin{equation*}
u^{\rho }(\cdot )=\bar{u}(\cdot )+\rho (u(\cdot )-\bar{u}(\cdot )).
\end{equation*}
The convexity of $U$ yields that $u^{\rho }(\cdot )\in \mathcal{U}$. Let $%
(x^{\rho }(\cdot ),y^{\rho }(\cdot ),z^{\rho }(\cdot ))$ be the
corresponding solutions of \eqref{state} and \eqref{y} associated with $%
u^{\rho }(\cdot )$.

Consider  the following variational equations \eqref{e-hat-x} and \eqref{e-hat-yz} along the optimal pair $(\bar{x}(\cdot ),\bar{u}(\cdot ))$: 
\begin{equation}\label{e-hat-x}
\left\{ 
\begin{aligned}
d\hat{x}(t)=&\,\, \Big\{A(t)\hat{x}(t)+b_{x}(t)\hat{x}(t)+b_{x_{\delta }}(t)\hat{%
x}_{\delta }(t)+b_{u}(t)(u(t)-\bar{u}(t)) \\
& \hspace{1em}+b_{u_{\delta }}(t)(u_{\delta }(t)-\bar{u}_{\delta }(t))\Big\}%
dt \\
& +\Big\{B(t)\hat{x}(t)+\sigma _{x}(t)\hat{x}(t)+\sigma _{x_{\delta }}(t)%
\hat{x}_{\delta }(t)+\sigma _{u}(t)(u(t)-\bar{u}(t)) \\
& \hspace{1em}+\sigma _{u_{\delta }}(t)(u_{\delta }(t)-\bar{u}_{\delta }(t))%
\Big\}dw(t),\,\,t\in \lbrack 0,T], \\
\hat{x}(t)=&\,\, 0,\,\,t\in \lbrack -\delta ,0],
\end{aligned}%
\right.
\end{equation}
\begin{equation}\label{e-hat-yz}
\left\{ 
\begin{aligned}
-d\hat{y}(t)=
&\,\,\mathbb{E}^{\mathcal{F}_{t}}\Big[\left\langle f_{x}(t),\hat{x}(t)\right\rangle
_{H}+\left\langle f_{x_{\delta }}(t),\hat{x}_{\delta }(t)\right\rangle
_{H}+f_{y}(t)\hat{y}(t)+f_{y_{_{\delta+}}}(t)\hat {y}_{\delta _{+}}(t)\\
&\hspace{1em}+\la f_{z}(t),\hat{z}(t)\ra_{\mathcal L_2^0(K;\R)} +\big< f_{z_{_{\delta
			+}}}(t),\hat {z}_{\delta _{+}}(t)\big>_{\mathcal L_2^0(K;\R)}\\
&\hspace{1em}+\left\langle f_{u}(t),u(t)-\bar{u}(t)\right\rangle _{H_{1}}+\left\langle
f_{u_{\delta }}(t),u_{\delta }(t)-\bar{u}_{\delta }(t)\right\rangle _{H_{1}}
\Big]dt \\
& -\hat{z}(t)dw(t),\,\,t\in \lbrack 0,T],\\
\hat{y}(T)= & \,\,\langle h_{x}(T),\hat{x}(T)\rangle _{H}+
\langle h_{x_{\delta,\nu}}(T),\hat
{x}_{\delta,\nu}(T)\rangle_H,\\
\hat{y}(t)=&\,\, 0,\,\hat{z}(t)=0,\,\,t\in (T,T+\delta ],
\end{aligned}
\right. 
\end{equation}
where we recall notations given by \eqref{Myeq3-3}, \eqref{e:eta-delta+} and \eqref{e:x=dv} for $x_{\delta}, y_{\delta_+}$ and $x_{\delta, \nu}$ respectively.

As $\rho$ goes to 0,  $u^\rho$ converges to $\bar u$,  and  formal calculations suggest that  $x^\rho$ (resp. $(y^\rho, z^\rho)$) converges to $\bar x$ (resp. $(\bar y, \bar z)$)  and $(x^\rho-\bar x)/\rho$ (resp. $ ((y^\rho-\bar y)/\rho, (z^\rho-\bar z)/\rho)$) converges to the solution $\hat x$ of  \eqref{e-hat-x} (resp. to the solution $(\hat y, \hat z)$ of \eqref{e-hat-yz}). This is justified by  Lemmas \ref{Myle4-1} and \ref{Le4-3} below. 
\begin{lemma}
\label{Myle4-1} Let (H1)-(H4) be satisfied. Then, we have, as $\rho\to 0$,
\begin{align*}
& \mathbb{E}\Big[ \sup\limits_{0\leq t\leq T}\Vert x^{\rho }(t)-\bar{x}%
(t)\Vert _{H}^{2}\Big] +\mathbb{E}\int_{0}^{T}\Vert x^{\rho }(t)-\bar{x}%
(t)\Vert _{V}^{2}dt=O(\rho ^{2}), \\
& \mathbb{E}\Big[ \sup\limits_{0\leq t\leq T}\Vert x^{\rho }(t)-\bar{x}%
(t)-\rho \hat{x}(t)\Vert _{H}^{2}\Big] +\E\int_0^T\Vert x^{\rho }(t)-\bar{x}%
(t)-\rho \hat{x}(t)\Vert _{V}^{2}dt=o(\rho ^{2}).
\end{align*}
\end{lemma}

\begin{proof}
We have%
\begin{equation*}
\left\{ 
\begin{aligned}
d(x^{\rho }(t)-\bar{x}(t))=& \,\,\Big\{A(t)(x^{\rho }(t)-\bar{x}(t))+\tilde{b}%
_{x}(t)(x^{\rho }(t)-\bar{x}(t))+\tilde{b}_{x_{\delta }}(t)(x_{\delta
}^{\rho }(t)-\bar{x}_{\delta }(t)) \\
& \hspace{1em}+\tilde{b}_{u}(t)\rho (u(t)-\bar{u}(t))+\tilde{b}_{u_{\delta }}(t)\rho
(u_{\delta }(t)-\bar{u}_{\delta }(t))\Big\}dt \\
& \hspace{0.5em}+\Big\{B(t)(x^{\rho }(t)-\bar{x}(t))+\tilde{\sigma}_{x}(t)(x^{\rho }(t)-%
\bar{x}(t))+\tilde{\sigma}_{x_{\delta }}(t)(x_{\delta }^{\rho }(t)-\bar{x}%
_{\delta }(t)) \\
&\hspace{1em} +\tilde{\sigma}_{u}(t)\rho (u(t)-\bar{u}(t))+\tilde{\sigma}_{u_{\delta
}}(t)\rho (u_{\delta }(t)-\bar{u}_{\delta }(t))\Big\}dw(t),\quad t\in
\lbrack 0,T], \\
x^\rho (t)-\bar x(t)=& \,\,0,\quad t\in \lbrack -\delta ,0],
\end{aligned}%
\right.
\end{equation*}%
where, for $\varphi =b,\sigma $ and $\tau =x,x_{\delta },u,u_{\delta }$,%
\begin{align*}
\tilde{\varphi}_{\tau }(t)=& \int_{0}^{1}\varphi _{\tau }\Big(t,\bar{x}
(t)+\lambda (x^{\rho }(t)-\bar{x}(t)),\bar{x}_{\delta }(t)+\lambda
(x_{\delta }^{\rho }(t)-\bar{x}_{\delta }(t)), \\
& \hspace{4em}\bar{u}(t)+\lambda \rho (u(t)-\bar{u}(t)),\bar{u}_{\delta
}(t)+\lambda \rho (u_{\delta }(t)-\bar{u}_{\delta }(t))\Big)d\lambda.
\end{align*}%
By Theorem \ref{estimate} and Lemma \ref{Le3-2}, we derive that 
\begin{align*}
& \mathbb{E}\Big[\sup\limits_{0\leq t\leq T}\Vert x^{\rho }(t)-\bar{x}%
(t)\Vert _{H}^{2}\Big]+\mathbb{E}\int_{0}^{T}\Vert x^{\rho }(t)-\bar{x}%
(t)\Vert _{V}^{2}dt \\
& \leq C\Big\{\mathbb{E}\int_{0}^{T}\Vert u^{\rho }(t)-\bar{u}(t)\Vert
_{H_{1}}^{2}dt+\mathbb{E}\int_{0}^{T}\Vert u_{\delta }^{\rho }(t)-\bar{u}%
_{\delta }(t)\Vert _{H_{1}}^{2}dt\Big\} \\
& \leq C\mathbb{E}\int_{0}^{T}\Vert u^{\rho }(t)-\bar{u}(t)\Vert
_{H_{1}}^{2}dt  \leq C\rho ^{2},
\end{align*}
and the first inequality follows.

Setting $\tilde{x}^{\rho }(t)=\frac{x^{\rho }(t)-\bar{x}(t)}{\rho }-\hat{x}%
(t)$, we have
\begin{equation}
\left\{ 
\begin{aligned}
d\tilde{x}^{\rho }(t)=&\,\, \Big\{A(t)\tilde{x}^{\rho }(t)+\tilde{b}_{x}(t)%
\tilde{x}^{\rho }(t)+[\tilde{b}_{x}(t)-b_{x}(t)]\hat{x}(t)+\tilde{b}%
_{x_{\delta }}(t)\tilde{x}_{\delta }^{\rho }(t) \\
& \hspace{1em}+[\tilde{b}_{x_{\delta }}(t)-b_{x_{\delta }}(t)]\hat{x}%
_{\delta }(t)+[\tilde{b}_{u}(t)-b_{u}(t)](u(t)-\bar{u}(t)) \\
& \hspace{1em}+[\tilde{b}_{u_{\delta }}(t)-b_{u_{\delta }}(t)](u_{\delta
}(t)-\bar{u}_{\delta }(t))\Big\}dt \\
& +\Big\{B(t)\tilde{x}^{\rho }(t)+\tilde{\sigma}_{x}(t)\tilde{x}^{\rho }(t)+[%
\tilde{\sigma}_{x}(t)-\sigma _{x}(t)]\hat{x}(t)+\sigma _{x_{\delta
}}^{\lambda }(t)\tilde{x}_{\delta }^{\rho }(t) \\
& \hspace{1em}+[\tilde{\sigma}_{x_{\delta }}(t)-\sigma _{x_{\delta }}(t)]%
\hat{x}_{\delta }(t)+[\tilde{\sigma}_{u}(t)-\sigma _{u}(t)](u(t)-\bar{u}(t))
\\
& \hspace{1em}+[\tilde{\sigma}_{u_{\delta }}(t)-\sigma _{u_{\delta
}}(t)](u_{\delta }(t)-\bar{u}_{\delta }(t))\Big\}dw(t), \\
\tilde{x}^{\rho }(t)=&\,\, 0,\,\,t\in \lbrack -\delta ,0].
\end{aligned}%
\right.
\end{equation}%
Utilizing the a prior estimate \eqref{es-x}, we have 
\begin{align*}
& \mathbb{E}\Big[\sup\limits_{0\leq t\leq T}\big\| \tilde{x}^{\rho }(t)\Vert
_{H}^{2}\Big]+\E\int_0^T\Vert \tilde{x}^{\rho }(t)\Vert _{V}^{2}dt\leq C\mathbb{E}\int_{0}^{T}\bigg\{\Big\| \lbrack \tilde{b}_{x}(t)-b_{x}(t)]\hat{x}(t)+[\tilde{b}_{x_{\delta }}(t)-b_{x_{\delta }}(t)]\hat{x}_{\delta }(t) \\
& \hspace{4.5em}+[\tilde{b}_{u}(t)-b_{u}(t)](u(t)-\bar{u}(t))+[\tilde{b}_{u_{\delta }}(t)-b_{u_{\delta }}(t)](u_{\delta }(t)-\bar{u}_{\delta}(t))\Big\|_{H}^{2} \\
& \hspace{4em}+\Big\| \lbrack \tilde{\sigma}_{x}(t)-\sigma _{x}(t)]\hat{x}%
(t)+[\tilde{\sigma}_{x_{\delta }}(t)-\sigma _{x_{\delta }}(t)]\hat{x}%
_{\delta }(t) \\
& \hspace{4.5em}+[\tilde{\sigma}_{u}(t)-\sigma _{u}(t)](u(t)-\bar{u}(t))+[%
\tilde{\sigma}_{u_{\delta }}(t)-\sigma _{u_{\delta }}(t)](u_{\delta }(t)-%
\bar{u}_{\delta }(t))\Big\|_{H}^{2}\bigg\}dt.
\end{align*}%
Then it follows from the dominated convergence theorem that 
\begin{equation*}
\lim\limits_{\rho \rightarrow 0}\mathbb{E}\Big[\sup\limits_{0\leq t\leq
T}\Vert \tilde{x}^{\rho }(t)\Vert _{H}^{2}\Big]+\lim\limits_{\rho \rightarrow 0}\E\int_0^T\Vert \tilde{x}^{\rho }(t)\Vert _{V}^{2}dt=0,
\end{equation*}
and the second inequality follows. 
\end{proof}

\begin{lemma}
\label{Le4-3} Assume (H1)-(H4) hold. Then
\begin{equation*}
\mathbb{E}\Big[\sup\limits_{0\leq t\leq T}|y^{\rho }(t)-\bar{y}(t)-\rho 
\hat{y}(t)| ^{2}\Big]+\mathbb{E}\int_{0}^{T}\|z^{\rho }(t)-\bar{z}(t)-\rho 
\hat{z}(t)\|^{2}_{\mathcal L_2^0(K,\R)}dt =o(\rho ^{2}).
\end{equation*}
\end{lemma}

\begin{proof}
It suffices to prove that 
\begin{equation*}
\lim\limits_{\rho \rightarrow 0}\mathbb{E}\left[ \sup\limits_{0\leq t\leq
T}\left\vert \frac{y^{\rho }(t)-\bar{y}(t)}{\rho }-\hat{y}(t)\right\vert ^{2}%
\right] +\lim\limits_{\rho \rightarrow 0}\mathbb{E}\int_{0}^{T}\|\frac{z^{\rho }(t)-\bar{z}(t)}{\rho}-
\hat{z}(t)\|^{2}_{\mathcal L_2^0(K,\R)}dt=0.
\end{equation*}%
Set $$\tilde{y}^{\rho }(t)=\frac{y^{\rho }(t)-\bar{y}(t)}{\rho }-\hat{y}(t)\ \text{and}\ \tilde{z}^{\rho }(t)=\frac{z^{\rho }(t)-\bar{z}(t)}{\rho }-\hat{z}(t).$$
Then 
\begin{equation}
\left\{ 
\begin{aligned}
-d\tilde{y}^{\rho }(t)=& \,\,\Big\{\big<\tilde{f}_{x}(t),\tilde{%
x}^{\rho }(t)\big>_{H}+\big<\tilde{f}_{x}(t)-f_{x}(t),\hat{x}(t)\big>_{H}+%
\big<\tilde{f}_{x_{\delta }}(t),\tilde{x}_{\delta }^{\rho }(t)\big>_{H} \\
& \hspace{0.5em}+\big<\tilde{f}_{x_{\delta }}(t)-f_{x_{\delta }}(t),\hat{x}%
_{\delta }(t)\big>_{H}+\tilde{f}_{y}(t)\tilde{y}^{\rho }(t)+[\tilde{f}%
_{y}(t)-f_{y}(t)]\hat{y}(t) \\
&\hspace{0.5em}+\tilde f_{y_{\delta+}}(t)\tilde y_{\delta+}^\rho(t)+[\tilde{f}
_{y_{\delta+}}(t)-f_{y_{\delta+}}(t)]\hat{y}_{\delta+}(t)+\big<\tilde{f}_{z}(t),\tilde{z}^{\rho }(t)\big>_{\mathcal L_2^0(K;\R)}\\
& \hspace{0.5em}+\big<\tilde{f}%
_{z}(t)-f_{z}(t),\hat{z}(t)\big>_{\mathcal L_2^0(K;\R)}+\big<\tilde{f}_{z_{\delta+}}(t),\tilde{z}_{\delta+}^{\rho }(t)\big>_{\mathcal L_2^0(K;\R)}\\&\hspace{0.5em}+\big<\tilde{f}%
_{z_{\delta+}}(t)-f_{z_{\delta+}}(t),\hat{z_{\delta+}}(t)\big>_{\mathcal L_2^0(K;\R)}+\big<\tilde{f}_{u}(t)-f_{u}(t),u(t)-\bar{u}(t)%
\big>_{H_{1}} \\
& \hspace{0.5em}+\big<\tilde{f}_{u_{\delta }}(t)-f_{u_{\delta
}}(t),u_{\delta }(t)-\bar{u}_{\delta }(t)\big>_{H_{1}}\Big\}dt-\tilde{z}%
^{\rho }(t)dw(t),\,\,t\in \lbrack 0,T], \\
\tilde{y}^{\rho }(T)=& \,\,\big<\tilde{h}_{x}(T),\tilde{x}^{\rho }(T)\big>%
_{H}+\big<\tilde{h}_{x}(T)-h_{x}(\bar{x}(T)),\hat{x}(T)\big>_{H}\\&+\big<\tilde{h}_{x_{\delta,\nu}}(T),\tilde{x}_{\delta,\nu}^{\rho }(T)\big>%
_{H}+\big<\tilde{h}_{x_{\delta,\nu}}(T)-h_{x_{\delta,\nu}}(T),\hat{x}_{\delta,\nu}(T)\big>_{H}
\end{aligned}%
\right.
\end{equation}%
where%
\begin{align*}
\tilde{h}_{x}(T)=&\int_{0}^{1}h_{x}\big(\bar{x}(T)+\lambda (x^{\rho }(T)-\bar{%
x}(T)),\bar{x}_{\delta,\nu}(T)+\lambda (x_{\delta,\nu}^{\rho }(T)-\bar{%
x}_{\delta,\nu}(T))\big)d\lambda ,\\
\tilde{h}_{x_{\delta,\nu}}(T)=&\int_{0}^{1}h_{x_{\delta,\nu}}\big(\bar{x}(T)+\lambda (x^{\rho }(T)-\bar{%
x}(T)),\bar{x}_{\delta,\nu}(T)+\lambda (x_{\delta,\nu}^{\rho }(T)-\bar{%
x}_{\delta,\nu}(T))\big)d\lambda ,
\end{align*}%
and for $\tau =x,x_{\delta },y,y_{\delta+},z,z_{\delta+},u,u_{\delta }$, 
\begin{align*}
\tilde{f}_{\tau }(t)=& \int_{0}^{1}f_{\tau }\big(t,\bar{x}(t)+\lambda
(x^{\rho }(t)-\bar{x}(t)),\bar{x}_{\delta }(t)+\lambda (x_{\delta }^{\rho
}(t)-\bar{x}_{\delta }(t)), \\
& \hspace{3em}\bar{y}(t)+\lambda (y^{\rho }(t)-\bar{y}(t)),\bar{y}_{\delta +}(t)+\lambda (y^{\rho }_{\delta +}(t)-\bar{y}_{\delta +}(t)),\\&\hspace{3em}\bar{z}(t)+\lambda (z^{\rho }(t)-\bar{z}(t)),\bar{z}_{\delta+
}(t)+\lambda (z_{\delta+ }^{\rho }(t)-\bar{z}_{\delta +}(t)), \\
& \hspace{3em}\bar{u}(t)+\lambda \rho (u(t)-\bar{u}(t)),\bar{u}_{\delta
}(t)+\lambda \rho (u_{\delta }(t)-\bar{u}_{\delta }(t))\big)d\lambda .
\end{align*}%
From the a priori estimate of classical BSDEs, Lemma \ref{Myle4-1} and Lemma
\ref{Le3-2}, we have
\begin{align*}
& \mathbb{E}\Big[\sup\limits_{0\leq t\leq T}|\tilde{y}^{\rho }(t)|^{2}%
\Big]+\mathbb{E}\int_{0}^{T}\|\tilde{z}^{\rho }(t)\|^{2}_{\mathcal L_2^0(K,\R)}dt \\
& \leq C\bigg(\mathbb{E}\int_{0}^{T}\Big\{\Big|\big<\tilde{f}_{x}(t),\tilde{x%
}^{\rho }(t)\big>_{H}+\big<\tilde{f}_{x}(t)-f_{x}(t),\hat{x}(t)\big>_{H}+%
\big<\tilde{f}_{x_{\delta }}(t),\tilde{x}_{\delta }^{\rho }(t)\big>_{H}+\big<%
\tilde{f}_{x_{\delta }}(t)-f_{x_{\delta }}(t),\hat{x}_{\delta }(t)\big>_{H}
\\
& \hspace{5em}+[\tilde{f}_{y}(t)-f_{y}(t)]\hat{y}(t)+[\tilde{f}_{y_{\delta+}}(t)-f_{y_{\delta+}}(t)]\hat{y}_{\delta+}(t)+ \langle \tilde{f}
_{z}(t)-f_{z}(t),\hat{z}(t)\rangle_{\mathcal L_2^0(K;\R)}\\&\hspace{5em}+ \langle \tilde{f}
_{z_{\delta+}}(t)-f_{z_{\delta+}}(t),\hat{z}_{\delta+}(t)\rangle_{\mathcal L_2^0(K;\R)}+\big<\tilde{f}_{u}(t)-f_{u}(t),u(t)-\bar{u}(t)
\big>_{H_{1}} \\
& \hspace{5em}+\big<\tilde{f}_{u_{\delta }}(t)-f_{u_{\delta }}(t),u_{\delta
}(t)-\bar{u}_{\delta }(t)\big>_{H_{1}}\Big|^{2}\Big\}dt \\
& \hspace{3em}+\mathbb{E}\Big[\Vert \widetilde{x}^{\rho }(T)\Vert _{H}^{2}+
\big|\big<\tilde{h}_{x}(T)-h_{x}(\bar{x}(T)),\hat{x}(T)\big>_{H}\big|^{2}+\big|\big<\tilde{h}_{x_{\delta,\nu}}(T)-h_{x_{\delta,\nu}}(T),\hat{x}_{\delta,\nu}(T)\big>_{H}\big|^2
\Big]\bigg)\rightarrow 0,
\end{align*}%
as $\rho \rightarrow 0$. The proof is complete.
\end{proof}

\subsection{Maximum principle}

In this subsection, we prove the stochastic maximum principle.

\begin{theorem}
\label{SMP} Suppose that (H1)-(H4) hold. Let $\bar{u}(\cdot )$ be an optimal
control for the control problem, $\bar{x}(\cdot )$ and $(\bar{y}(\cdot ),%
\bar{z}(\cdot ))$ be the corresponding solutions to \eqref{state} and %
\eqref{y}, respectively. Assume that $(p(\cdot ),q(\cdot ))$ is the solution
of (\ref{ade-1}) with $k(\cdot )$ being the solution of $\eqref{k}$. Then 
\begin{equation}
\bigg<H_{u}(t)+\mathbb{E}^{\mathcal{F}_{t}}\bigg[\int_{-\delta
}^{0}H_{u_{\delta }}(t-s)m(ds)\bigg],u-\bar{u}(t)\bigg>_{H_{1}}\geq 0,
\label{ne}
\end{equation}%
holds for all $u\in U$ and $dt\times dP$-almost all  $(t,\omega)\in[0,T]\times \Omega$, where we recall the notations \eqref{Hamiltionian} and  \eqref{e:H-H-tau} for the Hamiltonian function $H$.
\end{theorem}

The following result plays an important role in the derivation of the maximum principle.

\begin{lemma}
\label{Le3-1} Given a process $g\in L_{\mathbb{F}}^{2}(0,T+\delta ;E)$
satisfying $g(t)=0$ for $t\in (T,T+\delta ]$, and a process $\eta \in L_{%
\mathbb{F}}^{2}(-\delta ,T;E)$ satisfying $\eta (t)=0$ for $t\in \lbrack -\delta ,0)$, we have
\begin{equation}
\mathbb{E}\int_{0}^{T}\big<\eta _{\delta }(t),g(t)\big>_{E}dt=\mathbb{E}\int_{0}^{T}\big<\eta (t),\mathbb{E}^{\mathcal{F}_{t}}\big[g_{\delta _{+}}(t)\big]\big>_{E}dt.  \notag
\end{equation}
\end{lemma}

\begin{proof}
By \eqref{Myeq3-3} and the Fubini's theorem, we have
\begin{align*}
& \mathbb{E}\int_{0}^{T}\left\langle \eta _{\delta }(t),g(t)\right\rangle
_{E}dt =\mathbb{E}\int_{0}^{T}\Big<\int_{-\delta }^{0}\eta (t+s)m(ds),g(t)\Big>_{E}dt \\
& =\mathbb{E}\int_{0}^{T}\int_{-\delta }^{0}\big<g(t),\eta (t+s)\big>_{E}m(ds)dt 
 =\mathbb{E}\int_{-\delta }^{0}\int_{0}^{T}\big<g(t),\eta (t+s)\big>%
_{E}dtm(ds) \\
& =\mathbb{E}\int_{-\delta }^{0}\int_{s}^{T+s}\big<g(t-s),\eta (t)\big>_{E}dtm(ds),
\end{align*}
and further by Fubini's theorem again , we have
\begin{align*}
&\mathbb{E}\int_{-\delta }^{0}\int_{s}^{T+s}\big<g(t-s),\eta (t)\big>_{E}dtm(ds)\\
=&\mathbb{E}\int_{[-\delta ,0)}\int_{-\delta }^{t}\big<
g(t-s),\eta (t)\big>_{E}m(ds)dt +\mathbb{E}\int_{0}^{T-\delta }\int_{-\delta }^{0}\big<g(t-s),\eta (t)\big>_{E}m(ds)dt \\
& +\mathbb{E}\int_{(T-\delta ,T]}\int_{t-T}^{0}\big<g(t-s),\eta (t)\big>_{E}m(ds)dt =:I_{1}+I_{2}+I_{3}.
\end{align*}

Noting that $\eta (t)=0$ for $t<0$, we know that $I_{1}=0.$
Moreover, for $t\in (T-\delta ,T]$, if $s<t-T,$ then $T<t-s$, which implies $%
g(t-s)=0$ and thus, 
\begin{equation*}
\int_{\lbrack -\delta ,t-T)}\big<g(t-s),\eta (t)\big>_{E}m(ds)=0.
\end{equation*}
From this we get 
\begin{equation*}
I_{3}=\mathbb{E}\int_{(T-\delta ,T]}\int_{t-T}^{0}\big<g(t-s),\eta (t)\big>%
_{E}m(ds)dt=\mathbb{E}\int_{(T-\delta ,T]}\int_{-\delta }^{0}\big<%
g(t-s),\eta (t)\big>_{E}m(ds)dt=:I_3'.
\end{equation*}%
Therefore, we have
\begin{align*}
&\mathbb{E}\int_{0}^{T}\left\langle \eta _{\delta }(t),g(t)\right\rangle
_{E}dt= I_{2}+I_{3}=I_{2}+I'_{3} \\
=& \mathbb{E}\int_{0}^{T}\int_{-\delta }^{0}\big<g(t-s),\eta (t)\big>_{E}m(ds)dt 
= \mathbb{E}\int_{0}^{T}\bigg<\int_{-\delta }^{0}g(t-s)m(ds),\eta (t)\bigg>_{E}dt \\
=& \mathbb{E}\int_{0}^{T}\bigg<\mathbb{E}^{\mathcal{F}_{t}}\Big[\int_{-\delta}^{0}g(t-s)m(ds)\big],\eta (t)\bigg>_{E}dt=\mathbb{E}\int_{0}^{T}\big<\mathbb{E}^{\mathcal{F}_{t}}\big[g_{\delta _{+}}(t)\big], \eta (t)\big>_{E}dt,
\end{align*}
recalling the definition \eqref{e:eta-delta+} for $g_{\delta_+}$.
\end{proof}

\begin{proof}[Proof of Theorem \protect\ref{SMP}]
Recalling \eqref{ade-1} and \eqref{e-hat-x},  applying It\^{o}'s formula to $\left\langle p(t),\hat{x}(t)\right\rangle
_{H}$, and then taking expectation, we have 
\begin{align*}
& -\mathbb{E}\left[ k(T)\left\langle h_{x}(\bar{x}(T)),\hat{x}
(T)\right\rangle _{H}\right]-\E \int_{T-\delta}^T\big<h_{x_{\delta,\nu}}(T)k(T),\hat x(t)\big>_H\nu(d(t-T))\\
&=\mathbb{E}\int_{0}^{T}\bigg\{\big<b_{x_{\delta }}(t)\hat{x}_{\delta
}(t)+b_{u}(t)(u(t)-\bar{u}(t))+b_{u_{\delta }}(t)(u_{\delta }(t)-\bar{u}%
_{\delta }(t)),p(t)\big>_{H} \\
&\hspace{3em} -\bigg<\mathbb{E}^{\mathcal{F}_{t}}\Big[\int_{-\delta }^{0}b_{x_{\delta
}}^{\ast }(t-s)p(t-s)m(ds)+\int_{-\delta }^{0}\sigma _{x_{\delta }}^{\ast
}(t-s)q(t-s)m(ds) \\
&\hspace{7em} -\int_{-\delta }^{0}f_{x_{\delta }}(t-s)k(t-s)m(ds)\Big]-f_{x}(t)k(t),\hat{%
x}(t)\bigg>_{H} \\
& \hspace{3em}+\big<\sigma _{x_{\delta }}(t)\hat{x}_{\delta }(t)+\sigma _{u}(t)(u(t)-%
\bar{u}(t))+\sigma _{u_{\delta }}(t)(u_{\delta }(t)-\bar{u}_{\delta
}(t)),q(t)\big>_{{\mathcal L_2^0}}\bigg\}dt.
\end{align*}%
Noting  $\hat{x}(t) =0$ for  $t\in[-\delta, 0]$ and $p(t)=0$ for $t\in(T,T+\delta]$, we can apply Lemma \ref{Le3-1} and get
\begin{align*}
& \mathbb{E}\int_{0}^{T}\big<b_{x_{\delta }}(t)\hat{x}_{\delta }(t),p(t)\big>_{H}dt=\mathbb{E}\int_{0}^{T}\Big<\mathbb{E}^{\mathcal{F}_{t}}\Big[\int_{-\delta }^{0}b_{x_{\delta }}^{\ast }(t-s)p(t-s)m(ds)\Big],\hat{x}(t)\Big>_{H}dt, \\
& \mathbb{E}\int_{0}^{T}\big<\sigma _{x_{\delta }}(t)\hat{x}_{\delta
}(t),q(t)\big>_{\mathcal{L}_{2}^{0}}dt=\mathbb{E}\int_{0}^{T}\Big<\mathbb{E}^{\mathcal{F}_{t}}\Big[\int_{-\delta }^{0}\sigma _{x_{\delta }}^{\ast}(t-s)q(t-s)m(ds)\Big],\hat{x}(t)\Big>_{{ H}}dt, \\
& \mathbb{E}\int_{0}^{T}\big<f_{x_{\delta }}(t)k(t),\hat{x}_{\delta }(t)\big>%
_{H}dt=\mathbb{E}\int_{0}^{T}\Big<\mathbb{E}^{\mathcal{F}_{t}}\Big[%
\int_{-\delta }^{0}f_{x_{\delta }}(t-s)k(t-s)m(ds)\Big],\hat{x}(t)\Big>_{H}dt, \\
& \mathbb{E}\int_{0}^{T}\big<b_{u_{\delta }}(t)(u_{\delta }(t)-\bar{u}%
_{\delta }(t)),p(t)\big>_{H}dt \\
& \ \ \ \ =\mathbb{E}\int_{0}^{T}\Big<\mathbb{E}^{\mathcal{F}_{t}}\Big[%
\int_{-\delta }^{0}b_{u_{\delta }}^{\ast }(t-s)p(t-s)m(ds)\Big],u(t)-\bar{u}(t)\Big>_{H_{1}}dt, \\
& \mathbb{E}\int_{0}^{T}\big<\sigma _{u_{\delta }}(t)(u_{\delta }(t)-\bar{u}%
_{\delta }(t)),q(t)\big>_{\mathcal{L}_{2}^{0}}dt \\
& \ \ \ \ =\mathbb{E}\int_{0}^{T}\Big<\mathbb{E}^{\mathcal{F}_{t}}\Big[\int_{-\delta }^{0}\sigma _{u_{\delta }}^{\ast }(t-s)q(t-s)m(ds)\Big],u(t)-\bar{u}(t)\Big>_{H_{1}}dt.
\end{align*}%
Consequently, recalling the notation given by \eqref{e:x=dv} we have 
\begin{equation}
\begin{split}
& -\mathbb{E}\left[ k(T)\left\langle h_{x}(\bar{x}(T),\hat{x}
(T))\right\rangle _{H}\right]-\E [k(T) \langle h_{x_{\delta, \nu}}(T), \hat x_{\delta, \nu} (T)\rangle_H] \\ &=\mathbb{E}\int_{0}^{T}\Bigg\{\big<f_{x}(t)k(t),\hat{x}(t)\big>_{H}+\big<f_{x_{\delta }}(t)k(t),\hat{x}_{\delta
}(t)\big>_{H} +\Big<b_{u}^{\ast }(t)p(t)+\sigma _{u}^{\ast }(t)q(t)\\
&\quad +\mathbb{E}^{\mathcal{F}_{t}}\Big[\int_{-\delta }^{0}b_{u_{\delta }}^{\ast
}(t-s)p(t-s)m(ds) +\int_{-\delta }^{0}\sigma _{u_{\delta }}^{\ast
}(t-s)q(t-s)m(ds)\Big],u(t)-\bar{u}(t)\Big>_{H_{1}}\Bigg\}dt.
\end{split}
\label{eq:32}
\end{equation}
Applying It\^{o}'s formula to $k(t)\hat{y}(t)$ on $[0,T]$, we obtain by
Lemma \ref{Le3-1} that 
\begin{align}\label{eq:33}
 \hat{y}(0)&=-\mathbb{E}\int_{0}^{T}\Big\{\left\langle f_{x}(t),\hat{x}(t)\right\rangle _{H}+\left\langle f_{x_{\delta }}(t),\hat{x}_{\delta
}(t)\right\rangle _{H}+\left\langle f_{u}(t),u(t)-\bar{u}(t)\right\rangle
_{H_{1}} \nonumber\\
& \hspace{5em}+\left\langle f_{u_{\delta }}(t),u_{\delta }(t)-\bar{u}%
_{\delta }(t)\right\rangle _{H_{1}}\Big\}k(t)dt-\mathbb{E}%
[k(T)\hat y(T)] \nonumber \\
& =-\mathbb{E}\int_{0}^{T}\Big\{\left\langle f_{x}(t),\hat{x}%
(t)\right\rangle _{H}+\left\langle f_{x_{\delta }}(t),\hat{x}_{\delta
}(t)\right\rangle _{H}+\left\langle f_{u}(t),u(t)-\bar{u}(t)\right\rangle
_{H_{1}}\Big\}k(t)dt \\
& \hspace{1em}-\mathbb{E}\int_{0}^{T}\Big<\mathbb{E}^{\mathcal{F}_{t}}\int_{-\delta
}^{0}f_{u_{\delta }}(t-s)k(t-s)m(ds),u(t)-\bar{u}(t)\Big>_{H_{1}}dt\nonumber\\&\hspace{1em}-\mathbb{E}[k(T)\left\langle h_{x}(\bar{x}(T),\hat{x}(T))\right\rangle _{H}]-\mathbb{E}[k(T)\big<h_{x_{\delta,\nu}}(T),\hat{x}_{\delta,\nu}(T))\big> _{H}].\nonumber 
\end{align}
Then, it follows form combining \eqref{eq:32} with \eqref{eq:33} that 
\begin{align}\label{Myeq4-11}
\hat{y}(0)=& \mathbb{E}\int_{0}^{T}\big<b_{u}^{\ast }(t)p(t)+\sigma
_{u}^{\ast }(t)q(t)-f_{u}(t)k(t)\nonumber \\
& \hspace{1em}+\mathbb{E}^{\mathcal{F}_{t}}\big[\int_{-\delta
}^{0}b_{u_{\delta }}^{\ast }(t-s)p(t-s)m(ds)+\int_{-\delta }^{0}\sigma
_{u_{\delta }}^{\ast }(t-s)q(t-s)m(ds) \\
& \hspace{3em}-\int_{-\delta }^{0}f_{u_{\delta }}(t-s)k(t-s)m(ds)\big],u(t)-%
\bar{u}(t)\big>_{H_{1}}dt.\nonumber
\end{align}
On the other hand, by Lemma \ref{Le4-3} and the optimization assumption on $\bar{u}$, we know that
\[0\leq J(u^{\rho }(\cdot ))-J(\bar{u}(\cdot
))=\rho \hat{y}(0)+o(\rho ).\] 
This together with (\ref{Myeq4-11})  implies 
\begin{equation*}
\hat{y}(0)=\mathbb{E}\left[ \int_{0}^{T}\Big<H_{u}(t)+\mathbb{E}^{\mathcal{F}_{t}}\Big[\int_{-\delta }^{0}H_{u_{\delta }}(t-s)m(ds)\Big],u(t)-\bar{u}(t)\Big>_{H_{1}}dt\right] \geq 0,
\end{equation*}
from which  we obtain the maximum principle \eqref{ne}. 
\end{proof}
\begin{remark}
 In the proof above, $p(t)$ may not be continuous while $\hat{x}(t)$ is. So the possible jumps of $p$ do not contribute when applying  It\^{o}'s formula to $\left\langle p(t),\hat{x}(t)\right\rangle_{H}$.
\end{remark}

\begin{remark}
\label{Extension Rem} In view of the results on SEEs and ABSDEs established in the preceding sections, some straightforward adaptions of the proof of  Theorem \ref{SMP} shall yield a variety of extensions. We list some possible directions below. 

\begin{itemize}
\item[(i)] The delay measure $m(ds)$ appearing in the SEE \eqref{EQ-1} and the ABSDE \eqref{Myeq2-17} can be distinct. More precisely, the state equation may take the form
\begin{equation*}
\begin{cases}
& \displaystyle dx(t)= \left[A(t)x(t)+b(t,x(t),\int_{-\delta
}^{0}x(t+s)m_{1}(ds),u(t),\int_{-\delta }^{0}u(t+s)m_{2}(ds))\right]dt \\ 
& \ \ \displaystyle +\left[B(t)x(t)+\sigma (t,x(t),\int_{-\delta
}^{0}x(t+s)m_{3}(ds),u(t),\int_{-\delta }^{0}u(t+s)m_{4}(ds))\right]dw(t),\text{ }%
t\in \lbrack 0,T], \\ 
& x(t)= x_{0}(t),\,\,u(t)=v(t),\,\,\,t\in \lbrack -\delta ,0],%
\end{cases}
\end{equation*}%
and the cost functional $J(u(\cdot )):=y(0)$ is   determined by the ABSDE 
\begin{equation*}
	\left\{
	\begin{aligned}
		-dy(t)= &\,\, \mathbb E^{\mathcal F_t}\Bigg[f\Bigg(t,x(t),\int_{-\delta
		}^{0}x(t+s)m_{5}(ds),y(t),\int_{-\delta
		}^{0}y(t-s)m_{6}(ds),z(t),\int_{-\delta
		}^{0}z(t-s)m_{7}(ds),\\
		&\qquad \qquad \quad  u(t),\int_{-\delta }^{0}u(t+s)m_{8}(ds)\Bigg)\Bigg] dt   -z(t)dw(t),\,\,t\in \lbrack 0,T], \\ 
		y(T)= &\,\,h\left(x(T),\int_{-\delta}^0x(T+s)\nu(ds)\right), \text{ }y(t)=0,\text{ }z(t)=0,\text{ }t\in (T,T+\delta ].
	\end{aligned}
	\right.
\end{equation*}%
where $m_{i},i=1,\cdots 8,$ are finite measures on $[-\delta ,0].$

\item[(ii)] The delay measure $m$ and $\nu$ can be vector-valued, i.e., they may take values in $\mathbb{R}^{d}$ with $d>1$. 

\item[(iii)] The delays in $x$ and $u$ defined by \eqref{Myeq3-3} can be inhomogeneous in time:
\begin{equation*}
\eta_{\delta }(t)=\int_{-\delta }^{0}\phi (t,s)\eta(t+s)m(ds),
\end{equation*}%
where $\phi (t,s):[0,T]\times \lbrack -\delta ,0]\rightarrow \mathbb{R}$ is
a bounded measurable function.
Correspondingly, the anticipation \eqref{e:eta-delta+} in $y$ and $z$ now given by 
\[\eta_{\delta_+}=\int_{-\delta}^0 \phi(t, s)\eta(t-s) m(ds),\]
and  the adjoint equation \eqref{ade-1} is 
\begin{equation*}
\left\{ 
\begin{aligned}
-dp(t)=& \Big\{A^{\ast }(t)p(t)+B^{\ast }(t)q(t)+H_{x}(t)+\mathbb{E}^{%
\mathcal{F}_{t}}\Big[\int_{-\delta }^{0}\phi (t,s)H_{x_{\delta }}(t-s)m(ds)\Big]\Big\}dt \\
& +\E^{\mathcal F_t}[k(T)h_{x_{\delta,\nu}(T)}] \nu(d(t-T))I_{[T-\delta, T](t)} dt  -q(t)dw(t),\,\,t\in \lbrack 0,T], \\
p(T)=& -h_{x}(\bar{x}(T))k(T), \\
p(t)=& 0,\,q(t)=0,\,\,t\in (T,T+\delta ].
\end{aligned}%
\right.
\end{equation*}%
Here $H$ is the Hamiltonian defined in \eqref{Hamiltionian} and the function 
$k$ is given as (\ref{k}). The stochastic maximum principle \eqref{ne} becomes
\begin{equation*}
\Big<H_{u}(t)+\mathbb{E}^{\mathcal{F}_{t}}\Big[\int_{-\delta }^{0}\phi
(t, s)H_{u_{\delta }}(t-s)m(ds)\Big],u-\bar{u}(t)\Big>_{H_{1}}\geq 0.
\end{equation*}

\item[(iv)] One can extend the  measure $m$  to a finite \emph{signed} measure and the measure $\nu$ to a process $F$  with finite variation satisfying $(B5)$ (i.e., $\nu(d(t-T)) I_{[T-\delta,T]}(t)$ is replaced by $dF(t)$).

\item[(v)] The solution $(y,z)$ of ABSDE (\ref{y})  can be nonzero after $T$. More generally, it can also depend on $\{x(t),t>T\}$, e.g., $y(t)=h_1(x(t),x_{\delta,\nu}(t)),z(t)=l(x(t),x_{\delta,\nu}(t))$, $t>T$.
\end{itemize}
\end{remark}

\subsection{Sufficient conditions}

In this subsection, we will show that the necessary condition \eqref{ne}  for an
optimal control  is also sufficient under some convexity conditions.

\begin{theorem}
\label{sufficient} Suppose that (H1)-(H4) hold. Let $\bar u(\cdot)\in 
\mathcal{U}$ and $\bar{x}(\cdot )$ and $(\bar{y}(\cdot ),\bar{z}(\cdot ))$
be the corresponding solutions of \eqref{state} and \eqref{y}, respectively.
Assume 
\begin{enumerate}
\item[(a)] $h(x,x_{\delta,\nu})$ is convex in $(x,x_{\delta,\nu})$;

\item[(b)] the Hamiltonian $H$ given in \eqref{Hamiltionian} is convex in $(x,x_{\delta},y,y_{\delta+},z,z_{\delta+},u,u_{\delta })$ for each $(t,\omega ,p,q,k)$;

\item[(c)] \eqref {ne} holds for all $u\in U$, a.e., a.s.;

\item[(d)] $k(T) \le  0$ a.s., where we recall that $k(t)$ is the solution of \eqref{k}.
\end{enumerate}
Then $\bar u(\cdot)$ is an optimal control.
\end{theorem}

\begin{remark}
Clearly,  if  \eqref{k} does not contain delay terms, i.e., $f_{y_{\delta_+}}=f_{z_{\delta_+}}=0$,    we have $k(T)\le 0$ a.s. noting that $k(0)=-1$.   We impose the condition $k(T)\le 0$, since the delay terms (in particular the delay in diffusion) may change the sign of the solution of linear equation (see, e.g, \cite[Example~3.3]{ymy08}). 



\end{remark}

\begin{proof}
For an arbitrarily chosen admissible control process $u(\cdot )\in \mathcal{U}$,
let $x^{u}(\cdot )$ and $(y^{u}(\cdot ),z^{u}(\cdot ))$ be the corresponding
solutions of \eqref{state} and \eqref{y}, respectively. We
denote, for $t\in \lbrack 0,T]$, 
\begin{align*}
b^{u}(t)=& b(t,x^{u}(t),x_{\delta }^{u}(t),u(t),u_{\delta }(t)), \\
\sigma ^{u}(t)=& \sigma (t,x^{u}(t),x_{\delta }^{u}(t),u(t),u_{\delta }(t)),
\\
f^{u}(t)=& f(t,x^{u}(t),x_{\delta }^{u}(t),y^{u}(t),y^{u}_{\delta+}(t),z^{u}(t),z_{\delta+}^{u}(t),u(t),u_{\delta
}(t)).
\end{align*}
Applying It\^{o}'s formula to $k(t)(y^{u}(t)-\bar{y}(t))$ and $\left\langle
p(t),x^{u}(t)-\bar{x}(t)\right\rangle _{H}$ on $[0,T]$,  we can derive by Lemma \ref{Le3-1} that 
\begin{align*}
& \mathbb{E}\big[k(T)(h(x^{u}(T),x^u_{\delta,\nu}(T))-h(\bar{x}(T),\bar {x}_{\delta,\nu}(T))\big]+y^{u}(0)-\bar{y}(0)\\&\hspace{1em}-\mathbb{E}\big[k(T)\left\langle h_{x}(T),x^{u}(T)-\bar{x}(T)\right\rangle _{H}\big]-\mathbb{E}\big[k(T)\big< h_{x_{\delta,\nu}}(T),x_{\delta,\nu}^{u}(T)-\bar{x}_{\delta,\nu}(T)\big> _{H}\big] \\
& =\mathbb{E}\int_{0}^{T}\bigg\{f_{y}(t)k(t)(y^{u}(t)-\bar{y}
(t))+(f_{y_{\delta+}}k)_{\delta}(t)(y^{u}(t)-\bar{y}%
(t))\\&\hspace{3.5em}+\big<f_{z}(t)k(t),z^{u}(t)-\bar{z}(t)\big>_{\mathcal L_2^0(K;\R)}+\big<(f_{z_{\delta+}}k)_{\delta}(t),z^{u}(t)-\bar{z}(t)\big>_{{ \mathcal L_2^0(K;\R)} }\\
& \hspace{3.5em}+\big<b^{u}(t)-b(t),p(t)\big>_{H}+\big<\sigma ^{u}(t)-\sigma
(t),q(t)\big>_{ \mathcal L_2^0}-(f^{u}(t)-f(t))k(t) \\
& \hspace{3.5em}-\Big<b_{x}^{\ast }(t)p(t)+\sigma _{x}^{\ast
}(t)q(t)-f_{x}(t)k(t)+\mathbb{E}^{\mathcal{F}_{t}}\Big[\int_{-\delta
}^{0}b_{x_{\delta }}^{\ast }(t-s)p(t-s)m(ds) \\
& \hspace{4em}+\int_{-\delta }^{0}\sigma _{x_{\delta }}^{\ast
}(t-s)q(t-s)m(ds)-\int_{-\delta }^{0}f_{x_{\delta }}(t-s)k(t-s)m(ds)\Big]
,x^{u}(t)-\bar{x}(t)\Big>_{H}\bigg\}dt \\
& =\mathbb{E}\int_{0}^{T}\bigg\{H^{u}(t)-H(t)-\Big<H_{x}(t)+\mathbb{E}^{\mathcal{F}_{t}}\Big[\int_{-\delta
}^{0}H_{x_{\delta }}(t-s)m(ds)\Big],x^{u}(t)-\bar{x}(t)\Big>_{H}\\
& \hspace{4em}-\Big(H_{y}(t)+\int_{-\delta}^0H_{y_{\delta+}}(t+s)m(ds)\Big)(y^{u}(t)-\bar{y}%
(t)) \\&\hspace{4em}-\Big<H_{z}(t)+\int_{-\delta}^0H_{z_{\delta+}}(t+s)m(ds),z^{u}(t)-\bar{z}(t)\Big>_{\mathcal L_2^0(K;\R)}\bigg\}dt \\
& =\mathbb{E}\int_{0}^{T}\bigg\{H^{u}(t)-H(t)-\big<H_{x}(t),x^{u}(t)-\bar{x}(t)\big>_{H}-\big<H_{x_{\delta
}}(t),x_{\delta }^{u}(t)-\bar{x}_{\delta }(t)\big>_{H}\\
& \hspace{4em}-H_{y}(t)(y^{u}(t)-\bar{y}%
(t))-\E^{\mathcal F_t}[H_{y_{\delta+}}(t)(y^u_{\delta+}(t)-\bar y_{\delta+}(t) ]\\&\hspace{4em}-\big<H_{z}(t),z^{u}(t)-\bar{z}(t)\big>_{\mathcal L_2^0(K;\R)}-\E^{\mathcal F_t}[\big<H_{z_{\delta+}}(t),z^u_{\delta+}(t)-\bar z_{\delta+}(t)\big>_{\mathcal L_2^0(K;\R)} ] \bigg\}dt\\
=:& A. 
\end{align*}
By the convexity of $H$, 
we know that 
\begin{equation}\label{e:le01}
    A - B \ge 0
\end{equation}
where 
\begin{equation}\label{e:le02}
\begin{aligned}
B= &\mathbb{E}\int_{0}^{T}\bigg\{\big<H_{u}(t),u(t)-\bar{u}(t)\big>%
_{H_{1}}+\big<H_{u_{\delta }}(t),u_{\delta }(t)-\bar{u}_{\delta }(t)\big>_{H_{1}}\bigg\}dt \\
=&\mathbb{E}\int_{0}^{T}\Big<H_{u}(t)+\int_{-\delta }^{0} H_{u_{\delta}}(t-s)m(ds),u(t)-\bar{u}(t)\Big>_{H_{1}}dt\ge 0,
\end{aligned}
\end{equation}
with the nonnegativity following by the assumption \eqref{ne}. Therefore, 
\begin{align*}
y^{u}(0)-\bar{y}(0) = &\,\, -\mathbb{E}[k(T)(h(x^{u}(T), x^u_{\delta, \nu}(T))-h(\bar{x}(T)),\bar x^u_{\delta, \nu}(T))]\\&\hspace{0.5em}+\mathbb{E}\big[k(T)\left\langle h_{x}(T),x^{u}(T)-\bar{x}(T)\right\rangle _{H}+\mathbb{E}\big[k(T)\big< h_{x_{\delta,\nu}}(T),x_{\delta,\nu}^{u}(T)-\bar{x}_{\delta,\nu}(T)\big> _{H}\big]\\
&\quad + (A-B)+B \ge 0,
\end{align*}
where the last step follows from the convexity of $h$  together with the assumption (d)  $K(T)\le 0$, \eqref{e:le01} and \eqref{e:le02}. This shows  $J(u(\cdot ))-J(\bar{u}(\cdot ))\geq 0$  which yields the optimality of $\bar{u}(\cdot )$.
\end{proof}

\section{Some applications} \label{sec:APP}

In this section, we apply our result to  optimal control problem of  parabolic SPDEs wiht delay and the linear quadratic (LQ) problem of SDEEs. 

\subsection{Optimal control problem of SPDEs with delay}

Let $H^{1}$ be the Sobolev space of $W^{1,2}(\mathbb{R}^{d}).$ Set $%
V=H^{1} $ and $H=L^{2}(\mathbb{R}^{d}).$ Consider super-parabolic SPDE with
delays: 
\begin{equation*}
\left\{ 
\begin{aligned}
{d}x(t,\zeta )& =\Big\{\sum_{i,j=1}^{n}\partial _{\zeta _{i}}[\alpha
_{ij}(t,\zeta )\partial _{\zeta _{j}}x(t,\zeta )]+\sum_{i=1}^{n}\tilde{\alpha%
}_{i}(t,\zeta )\partial _{\zeta _{i}}x(t,\zeta ) \\
& +b(t,\zeta ,u(t),u(t-\delta ),x(t,\zeta ),x(t-\delta ,\zeta ))\Big\}{d}t \\
& +\Big\{\sum_{i=1}^{n}\beta _{i}(t,\zeta )\partial _{\zeta _{i}}x(t,\zeta
)+\sigma (t,\zeta ,u(t),u(t-\delta ),x(t,\zeta ),x(t-\delta ,\zeta ))\Big\}{d%
}w(t), \\
& (t,\zeta )\in \lbrack 0,T]\times \mathbb{R}^{d}, \\
x(t,\zeta )& =x_{0}(t,\zeta ),\ u(t,\zeta )=v(t,\zeta ),\quad (t,\zeta )\in
\lbrack -\delta ,0]\times \mathbb{R}^{d}.
\end{aligned}%
\right.
\end{equation*}%
Here $\alpha _{ij},\tilde{\alpha}_{i},\beta _{i},b,\sigma $ and\ $(x_{0},v)$%
\ are given coefficients and initial values, respectively.\ The control $%
u(t) $ is a progressively measurable  process taking values in $U$ which is a convex subset of a
separable Hilbert space $H_{1}$. We consider the problem of minimizing the
cost functional 
\begin{equation*}
J(u(\cdot ))=y(0),
\end{equation*}%
where $y$ is the recursive utility subject to the following BSDE: 
\begin{equation*}
y(t)=\int_{\mathbb{R}^{d}}h(\zeta ,x(T,\zeta ))d\zeta +\int_{t}^{T}\int_{%
\mathbb{R}^{d}}f(s,\zeta ,y(s),z(s),u(s),x(s,\zeta ))d\zeta
ds-\int_{t}^{T}z(s)dw({s}).
\end{equation*}%
We assume the corresponding differentiation and growth conditions on the
coefficients $b,\sigma ,h$ and $f,$ and assume $x_{0}\in S_{\mathcal{F}%
}^{2}(-\delta ,0;H)$ and $v\in L_{\mathcal{F}}^{2}(-\delta ,0;H_{1})$. We take 
\begin{equation*}
A(t)=\sum_{i,j=1}^{n}\partial _{\zeta _{i}}[\alpha _{ij}(t,\zeta )\partial
_{\zeta _{j}}]+\sum_{i=1}^{n}\tilde{\alpha}_{i}(t,\zeta )\partial _{\zeta
_{i}},\quad B(t)=\sum_{i=1}^{n}\beta _{i}(t,\zeta )\partial _{\zeta _{i}},
\end{equation*}
and 
suppose that there exist constants $\alpha \in (0,1)$ and $K>1$ such that 
\begin{equation*}
\alpha I_{d\times d}+(\beta ^{i}\beta ^{j})_{d\times d}\leq 2(\alpha
^{ij})_{d\times d}\leq KI_{d\times d}.
\end{equation*}

In this case, we can check that the conditions $(A2)$ and $(A3)$ hold.
Thus, we obtain by Theorem~\ref{SMP} the maximum condition, and  by Theorem~\ref{sufficient} its sufficiency under proper convex
assumptions.

\subsection{LQ problem for SDEEs}

\label{4.1} Suppose that the control domain is a separable Hilbert space $H_{1}$ and $\mathcal{U}=L_{\mathcal{F}}^{2}(0,T;H_{1})$. In \eqref{state}
and \eqref{y}, let 
\begin{align*}
& b(t,x,x_{\delta },u,u_{\delta })=A_{1}(t)x_{\delta
}+C(t)u+C_{1}(t)u_{\delta }, \\
& \sigma (t,x,x_{\delta },u,u_{\delta })=B_{1}(t)x_{\delta
}+D(t)u+D_{1}(t)u_{\delta }, \\
& f(t,x,x_{\delta },y,z,u,u_{\delta })=\left\langle F(t)x,x\right\rangle
_{H}+{G}_{1}{(t)y+G}_{2}{(t)z}+\left\langle N(t)u,u\right\rangle
_{H_{1}}, \\
& h(x)=\left\langle \Phi x,x\right\rangle _{H},
\end{align*}%
where $A_{1},B_{1},F:[0,T]\times \Omega \rightarrow \mathcal{L}(H)$, $%
C,C_{1},D,D_{1}:[0,T]\times \Omega \rightarrow \mathcal{L}(H_{1},H)$, $%
G_{1},G_{2}:[0,T]\times \Omega \rightarrow \mathbb{R}$, $N:[0,T]\times
\Omega \rightarrow \mathcal{L}(H_{1})$, and $\Phi :\Omega \rightarrow 
\mathcal{L}(H)$.

Then, the controlled system is as follows 
\begin{equation}
\left\{ \begin{aligned}
dx(t)&=\big[A(t)x(t)+A_1(t)x_\delta(t)+C(t)u(t)+C_1(t)u_\delta
(t)\big]dt\\&\hspace{0.5em}+[B(t)x(t)+B_1(t)x_\delta(t)+D(t)u(t)+D_1(t)u_
\delta (t)]dw(t),\,t\in[0,T],\\
x(t)&=x_0(t),\,\,u(t)=v(t),\,\,t\in[-\delta,0], \end{aligned}\right.
\label{lqx}
\end{equation}%
and the recursive utility $y(\cdot )$ is governed by 
\begin{equation}
\left\{ \begin{aligned} dy(t)=&\,\,\Big\{\left\langle
F(t)x(t),x(t)\right\rangle_H+
G_1(t)y(t)+G_2(t)z(t)+\left\langle
N(t)u(t),u(t)\right\rangle_{H_1}\Big\}dt\\&-z(t)dw(t),\,\,\,t\in[0,T],\\
y(T)=&\,\,\left\langle\Phi x(T),x(T)\right\rangle_H. \end{aligned}\right.
\end{equation}
We aim to minimize $J(u(\cdot )):=y(0)$ over $\mathcal{U}$.

Assume the following conditions.

\begin{itemize}
\item[(\textbf{L1})] The operators $A:[0,T]\times \Omega \rightarrow 
\mathcal {L}(V;V^{\ast })$ and $B:[0,T]\times \Omega \rightarrow \mathcal {L}(V;H)$ satisfy the coercivity and boundedness conditions (A2)-(A3).

\item[(\textbf{L2})]  $x_{0}(\cdot )\in S_{\mathbb{F}%
}^{2}(-\delta ,0;H)$ and $v(\cdot )\in L_{\mathbb{F}}^{2}(-\delta
,0;H_{1}) $. The processes $A_{1},B_{1},C,C_{1},D,D_{1},F,G,H,N$ are
uniformly bounded, $A_{1},B_{1},C,C_{1},D,D_{1},F,N$ are weakly {$\mathbb{F}$%
-}adapted (see \cite[Section 2]{liu2021maximum} and \cite[Section 1]%
{du2013maximum} for details) and $
G_{1},G_{2} $ are $\mathbb{F}$-adapted. $\Phi $ is uniformly
bounded and weakly $\mathcal{F}_{T}$-measurable.

\item[(\textbf{L3})]  $F,$ $N$, $\Phi $ are symmetric and
nonnegative definite for a.e. $t\in \lbrack 0,T]$, $\mathbb{P}-$a.s. Furthermore, $N$
is uniformly positive  definite for a.e. $t\in \lbrack 0,T]$, $\mathbb{P}-$a.s.
\end{itemize}

The Hamiltonian becomes 
\begin{align*}
 H(t,x,x_{\delta },y,z,u,u_{\delta },p,q,k) 
& =\left\langle A_{1}(t)x_{\delta }+C(t)u+C_{1}(t)u_{\delta },p\right\rangle
_{H}  +\left\langle B_{1}(t)x_{\delta }+D(t)u+D_{1}(t)u_{\delta },q\right\rangle
_{H} \\
& -\left\langle F(t)x,x\right\rangle _{H}-G_{1}{(t)y-G}_{2}{(t)z%
}-\left\langle N(t)u,u\right\rangle _{H_{1}}.
\end{align*}
Assume that $\bar{u}(\cdot )$ is an optimal control and $\bar{x}(\cdot )$ is
the corresponding solution of equation \eqref{lqx}. Then, the adjoint
equation along $(\bar{x}(\cdot ),\bar{u}(\cdot ))$ is 
\begin{equation}
\left\{ \begin{aligned} -dp(t)=&\bigg\{A^*(t)p(t)+B^*(t)q(t)-2F(t)\bar
x(t)k(t)\\&+\mathbb E^{\mathcal
F_t}\Big[\int_{-\delta}^0A_1^*(t-s)p(t-s)m(ds)+\int_{-
\delta}^0B_1^*(t-s)q(t-s)m(ds)\Big]\bigg\}dt\\&-q(t)dw(t),\,\,\,t\in[0,T],\\
p(T)=&-2\Phi \bar x(T),\,\,p(t)=q(t)=0,\,\,t\in(T,T+\delta], \end{aligned}%
\right.
\end{equation}%
and $k(\cdot )$ satisfies 
\begin{equation}
\left\{ \begin{aligned} dk(t)=&\,\,
G(t)k(t)dt+H(t)k(t)dw(t),\,\,t\in[0,T],\\ k(0)=&\,\,-1. \end{aligned}\right.
\end{equation}%
Now \eqref{ne} in the maximum principle becomes 
\begin{align*}
& \big<C^{\ast }(t)p(t)+D^{\ast }(t)q(t)+\mathbb{E}^{\mathcal{F}_{t}}\Big[\int_{-\delta }^{0}C_{1}^{\ast }(t-s)p(t-s)m(ds) \\
&\quad  +\int_{-\delta }^{0}D_{1}^{\ast }(t-s)q(t-s)m(ds)\Big]-2N(t)\bar{u}(t),u-%
\bar{u}(t)\big>_{H_{1}}=0.
\end{align*}%
From this we can deduce that 
\begin{equation*}
C^{\ast }(t)p(t)+D^{\ast }(t)q(t)+\mathbb{E}^{\mathcal{F}_{t}}\Big[\int_{-\delta }^{0}C_{1}^{\ast }(t-s)p(t-s)m(ds)+\int_{-\delta
}^{0}D_{1}^{\ast }(t-s)q(t-s)m(ds)\Big]-2N(t)\bar{u}(t)=0,
\end{equation*}%
and thus 
\begin{align}
\bar{u}(t)=-& \frac{1}{2}N^{-1}\bigg\{C^{\ast }(t)p(t)+D^{\ast }(t)q(t)+%
\mathbb{E}^{\mathcal{F}_{t}}\Big[\int_{-\delta }^{0}C_{1}^{\ast
}(t-s)p(t-s)m(ds)  \notag \\
& \hspace{4em}+\int_{-\delta }^{0}D_{1}^{\ast }(t-s)q(t-s)m(ds)\Big]\bigg\}.
\end{align}%
By Theorem \ref{sufficient}, we can conclude that $\bar{u}(\cdot )$ defined
above is indeed an optimal control of the LQ problem.

\appendix

\section{Proof of Theorem \ref{exu} } \label{se:Thm3.2}

Consider the following subspace of $L_{\mathbb{F}}^{2}(-\delta ,T;V)$
(still equipped with the norm on $L_{\mathbb{F}}^{2}(-\delta ,T;V))$: 
\begin{equation*}
\mathcal{H}:=\left\{ y:y\in L_{\mathbb{F}}^{2}(-\delta ,T;V)\text{ such that}\ y(t)=x_{0}(t),\,t\in \lbrack -\delta ,0]\right\} .
\end{equation*}%
We introduce an equivalent norm on $\mathcal H$: for some undetermined $\beta ,\gamma
>0$, 
\begin{equation*}
\Vert y\Vert _{\beta ,\gamma }^{2}:=\int_{-\delta }^{0}\Vert y(s)\Vert
_{V}^{2}ds+\mathbb{E}\left[ \int_{0}^{T}e^{-\beta t}\Vert y(t)\Vert
_{H}^{2}dt+\gamma \int_{0}^{T}e^{-\beta t}\Vert y(t)\Vert _{V}^{2}dt\right] .
\end{equation*}

Given $X(\cdot )\in \mathcal{H}$, we consider the following equation 
\begin{equation}\label{x}
\left\{ 
\begin{aligned}
dx(t)=& [A(t)x(t)+b(t,x(t),X_{\delta }(t))]dt \\
& +[B(t)x(t)+\sigma (t,x(t),X_{\delta }(t))]dw(t),\ t\in \lbrack 0,T], \\
x(t)=& x_{0}(t),\text{ }t\in \lbrack -\delta ,0].
\end{aligned}%
\right.  
\end{equation}%
According to \cite{79k,LR15}, under (A1)-(A4), \eqref{x} has a unique
solution $x(\cdot )\in \mathcal{H}$. Define the mapping $\mathbb{I}$ from $%
\mathcal{H}$ to itself by $\mathbb{I}(X(\cdot ))=x(\cdot )$. Given any $%
X(\cdot ),X^{\prime }(\cdot )\in \mathcal{H}$, we denote $x(\cdot )=\mathbb{I%
}(X(\cdot )),x^{\prime }(\cdot )=\mathbb{I}(X^{\prime }(\cdot ))$. For $t\in
\lbrack -\delta ,T]$, we define 
\begin{equation*}
\hat{X}(t)=X(t)-X^{\prime }(t),\,\,\hat{x}(t)=x(t)-x^{\prime }(t).
\end{equation*}%
Then $\hat{x}(t)$ satisfies 
\begin{equation*}
\left\{ 
\begin{aligned}
d\hat{x}(t)=& [A(t)\hat{x}(t)+b(t,x(t),X_{\delta }(t))-b(t,x^{\prime
}(t),X_{\delta }^{\prime }(t))]dt \\
& +[B(t)\hat{x}(t)+\sigma (t,x(t),X_{\delta }(t))-\sigma (t,x^{\prime
}(t),X_{\delta }^{\prime }(t))]dw(t), \\
\hat{x}(t)=& 0,\,t\in \lbrack -\delta ,0].
\end{aligned}%
\right.
\end{equation*}%
For $\beta >0$, applying It\^{o}'s formula to $e^{-\beta t}\Vert \hat{x}%
(t)\Vert _{H}^{2}$ on $[0,T]$, we obtain 
\begin{align*}
& e^{-\beta T}\mathbb{E}[\Vert \hat{x}(T)\Vert _{H}^{2}]+\beta \mathbb{E}%
\int_{0}^{T}e^{-\beta t}\Vert \hat{x}(t)\Vert _{H}^{2}dt \\
& =2\mathbb{E}\int_{0}^{T}e^{-\beta t}\left\langle A(t)\hat{x}(t),\hat{x}%
(t)\right\rangle _{\ast }dt \\
& \hspace{1em}+2\mathbb{E}\int_{0}^{T}e^{-\beta t}\left\langle
b(t,x(t),X_{\delta }(t))-b(t,x^{\prime }(t),X_{\delta }^{\prime }(t)),\hat{x}%
(t)\right\rangle _{H}dt \\
& \hspace{1em}+\mathbb{E}\int_{0}^{T}e^{-\beta t}\Vert B(t)\hat{x}(t)+\sigma
(t,x(t),X_{\delta }(t))-\sigma (t,x^{\prime }(t),X_{\delta }^{\prime
}(t))\Vert _{\mathcal{L}_{2}^{0}}^{2}dt.
\end{align*}%

Then by the coercivity condition (A2), we can get  that
\begin{align*}
& e^{-\beta T}\mathbb{E}[\Vert \hat{x}(T)\Vert _{H}^{2}]+\beta \mathbb{E}%
\int_{0}^{T}e^{-\beta t}\Vert \hat{x}(t)\Vert _{H}^{2}dt \\
& =\mathbb{E}\int_{0}^{T}e^{-\beta t}[2\left\langle A(t)\hat{x}(t),\hat{x}%
(t)\right\rangle _{\ast }+\Vert B(t)\hat{x}(t)\Vert _{H}^{2}]dt \\
& \hspace{1em}+2\mathbb{E}\int_{0}^{T}e^{-\beta t}\left\langle B(t)\hat{x}%
(t),\sigma (t,x(t),X_{\delta }(t))-\sigma (t,x^{\prime }(t),X_{\delta
}^{\prime }(t))\right\rangle dt \\
& \hspace{1em}+2\mathbb{E}\int_{0}^{T}e^{-\beta t}\left\langle
b(t,x(t),X_{\delta }(t))-b(t,x^{\prime }(t),X_{\delta }^{\prime }(t)),\hat{x}%
(t)\right\rangle _{H}dt \\
& \hspace{1em}+\mathbb{E}\int_{0}^{T}e^{-\beta t}\Vert \sigma
(t,x(t),X_{\delta }(t))-\sigma (t,x^{\prime }(t),X_{\delta }^{\prime
}(t))\Vert _{\mathcal{L}_{2}^{0}}^{2}dt \\
& \leq \mathbb{E}\int_{0}^{T}e^{-\beta t}[-\alpha \Vert \hat{x}(t)\Vert
_{V}^{2}+\lambda \Vert \hat{x}(t)\Vert _{H}^{2}]dt \\
& \hspace{1em}+\frac{\alpha }{2}\mathbb{E}\int_{0}^{T}e^{-\beta t}\Vert \hat{%
x}(t)\Vert _{V}^{2}dt+\mathbb{E}\int_{0}^{T}e^{-\beta t}\Vert \hat{x}%
(t)\Vert _{H}^{2}dt \\
& \hspace{1em}+\mathbb{E}\int_{0}^{T}e^{-\beta t}\Vert b(t,x(t),X_{\delta
}(t))-b(t,x^{\prime }(t),X_{\delta }^{\prime }(t))\Vert _{H}^{2}dt \\
& \hspace{1em}+(1+\frac{2(C_{1})^{2}}{\alpha })\mathbb{E}\int_{0}^{T}e^{-%
\beta t}\Vert \sigma (t,x(t),X_{\delta }(t))-\sigma (t,x^{\prime
}(t),X_{\delta }^{\prime }(t))\Vert _{H}^{2}dt,
\end{align*}
where the constant $C_{1}$ is from  (\ref{Myeq2-41}). Thus, we get
\begin{align*}
& \beta \mathbb{E}\int_{0}^{T}e^{-\beta t}\Vert \hat{x}(t)\Vert _{H}^{2}dt+%
\frac{\alpha }{2}\mathbb{E}\int_{0}^{T}e^{-\beta t}\Vert \hat{x}(t)\Vert
_{V}^{2}dt \\
& \leq (1+\lambda )\mathbb{E}\int_{0}^{T}e^{-\beta t}\Vert \hat{x}(t)\Vert
_{H}^{2}dt+\mathbb{E}\int_{0}^{T}e^{-\beta t}\Vert b(t,x(t),X_{\delta
}(t))-b(t,x^{\prime }(t),X_{\delta }^{\prime }(t))\Vert _{H}^{2}dt \\
& \hspace{1em}+(1+\frac{2(C_{1})^{2}}{\alpha })\mathbb{E}\int_{0}^{T}e^{-%
\beta t}\Vert \sigma (t,x(t),X_{\delta }(t))-\sigma (t,x^{\prime
}(t),X_{\delta }^{\prime }(t))\Vert _{\mathcal{L}_{2}^{0}}^{2}dt.
\end{align*}%
Then by the Lipschitz condition (A4), we have 
\begin{align*}
& \big(\beta -\lambda -1-2(1+\tfrac{(C_{1})^{2}}{\alpha })K\big)\mathbb{E}%
\int_{0}^{T}e^{-\beta t}\Vert \hat{x}(t)\Vert _{H}^{2}dt+\frac{\alpha }{2}%
\mathbb{E}\int_{0}^{T}e^{-\beta t}\Vert \hat{x}(t)\Vert _{V}^{2}dt \\
& \hspace{2em}\leq 2(1+\tfrac{(C_{1})^{2}}{\alpha })K\mathbb{E}%
\int_{0}^{T}e^{-\beta t}\Vert \hat{X}_{\delta }(t)\Vert _{H}^{2}dt.
\end{align*}%
By Lemma \ref{Le3-2} and the fact that $\hat{X}(t)=0$ for $t\in \lbrack 
-\delta ,0)$, we have 
\begin{equation*}
\mathbb{E}\int_{0}^{T}e^{-\beta t}\Vert \hat{X}_{\delta }(t)\Vert
_{H}^{2}dt\leq m^{2}([-\delta ,0])\mathbb{E}\int_{-\delta }^{T}e^{-\beta
t}\Vert \hat{X}(t)\Vert _{H}^{2}dt=m^{2}([-\delta ,0])\mathbb{E}%
\int_{0}^{T}e^{-\beta t}\Vert \hat{X}(t)\Vert _{H}^{2}dt.
\end{equation*}%
Combining the above two inequalities, we get
\begin{equation}\label{e:est-SDEE}
\begin{aligned}
& \big(\beta -\lambda -1-2(1+\tfrac{(C_{1})^{2}}{\alpha })K\big)\mathbb{E}%
\int_{0}^{T}e^{-\beta t}\Vert \hat{x}(t)\Vert _{H}^{2}dt+\frac{\alpha }{2}%
\mathbb{E}\int_{0}^{T}e^{-\beta t}\Vert \hat{x}(t)\Vert _{V}^{2}dt \\
& \hspace{2em}\leq 2(1+\tfrac{(C_{1})^{2}}{\alpha })Km^2([-\delta, 0])\mathbb{E}
\int_{0}^{T}e^{-\beta t}\Vert \hat{X}(t)\Vert _{H}^{2}dt.
\end{aligned}
\end{equation}

We may assume $m([-\delta, 0])>0$, since if otherwise $m([-\delta, 0])=0$, we have  $X_\delta=X'_\delta\equiv 0$ and hence \eqref{x} reduces to the classical SEE without delay. Now, if we choose
\begin{equation*}
\beta =\lambda +1+2(1+\tfrac{(C_{1})^{2}}{\alpha })K+4(1+\tfrac{C_{1}}{%
\alpha })Km^{2}([-\delta ,0]),
\end{equation*}%
by \eqref{e:est-SDEE} we deduce
\begin{align*}
& \mathbb{E}\int_{0}^{T}e^{-\beta t}\Vert \hat{x}(t)\Vert _{H}^{2}dt+\frac{%
\alpha }{8(1+\tfrac{(C_{1})^{2}}{\alpha })Km^{2}([-\delta ,0])}\mathbb{E}%
\int_{0}^{T}e^{-\beta t}\Vert \hat{x}(t)\Vert _{V}^{2}dt\leq \frac{1}{2}%
\mathbb{E}\int_{0}^{T}e^{-\beta t}\Vert \hat{X}(t)\Vert _{H}^{2}dt \\
& \leq \frac{1}{2}\bigg( \mathbb{E}\int_{0}^{T}e^{-\beta t}\Vert \hat{X}%
(t)\Vert _{H}^{2}dt+\frac{\alpha }{8(1+\tfrac{(C_{1})^{2}}{\alpha }%
)Km^{2}([-\delta ,0])}\mathbb{E}\int_{0}^{T}e^{-\beta t}\Vert \hat{X}%
(t)\Vert _{V}^{2}dt\bigg),
\end{align*}
that is, \[\Vert \hat x \Vert
_{\beta ,\gamma }^{2} = \Vert \mathbb I(X) - \mathbb I (X')\Vert
_{\beta ,\gamma }^{2} \le \frac12 \Vert X-X'\Vert
_{\beta ,\gamma }^{2}=\frac12 \Vert \hat X\Vert
_{\beta ,\gamma }^{2}\] with $\gamma =\frac{\alpha }{8(1+\tfrac{(C_{1})^{2}}{\alpha })Km^{2}([-\delta ,0])}.$
Thus $\mathbb{I}$ is a contraction mapping on $\mathcal H$ with the norm $\Vert \cdot \Vert
_{\beta ,\gamma }$
and hence by the contraction mapping theorem, there exists a unique solution in $\mathcal{H}$ to equation \eqref{sdee}. This completes the proof.

\noindent\textbf{Acknowledgments.} We wish to thank 
 Ying Hu, Wei Liu, and Fengyu Wang for helpful discussions and comments. G. Liu's Research is partially supported by National Natural Science Foundation of
China (No. 12201315) and the Fundamental Research Funds for the
Central Universities, Nankai University (No. 63221036). J. Song is partially supported by National Natural Science Foundation of China (nos. 12071256 and 12226001) and Major Basic Research Program of the Natural Science Foundation of Shandong Province in China (nos. ZR2019ZD42 and ZR2020ZD24).

\end{document}